\documentclass[twoside, final]{amsart}

\usepackage[margin=1in]{geometry}

\usepackage{amsfonts, amssymb, amsmath, amsthm, multicol,bbm}
\usepackage[colorlinks=true, pdfstartview=FitV, linkcolor=blue,
            citecolor=blue, urlcolor=blue]{hyperref}
\usepackage[usenames]{color}
\definecolor{Red}{rgb}{0.7,0,0.1}
\definecolor{Green}{rgb}{0,0.7,0}
\usepackage{accents}
\usepackage{comment,url}
\usepackage[all]{xy}

\usepackage{mathtools}

\mathtoolsset{showonlyrefs}
\usepackage[shortlabels]{enumitem}

\usepackage[notref,notcite,color]{showkeys}
\definecolor{labelkey}{rgb}{0,0,1}

\numberwithin{equation}{section}

\newtheorem{Thm}{Theorem}[section]
\newtheorem{Lem}[Thm]{Lemma}
\newtheorem{Prop}[Thm]{Proposition}
\newtheorem{Cor}[Thm]{Corollary}

\newtheorem{Rmk}[Thm]{Remark}

\newtheorem*{Thm*}{Theorem}

\usepackage{cjhebrew}

\newcommand{\E}{\mathbb{E}}
\newcommand{\Prb}{\mathbb{P}}
\newcommand{\RR}{\mathbb{R}}
\newcommand{\TT}{\mathbb{T}^2}
\newcommand{\al}{\alpha}

\newcommand{\pg}{>}

\newcommand{\bfU}{\mathbf{u}}

\newcommand{\Trs}{\theta}

\usepackage[numbers,sort&compress]{natbib}

\begin{document}

\author[F\" oldes and Sy]{Juraj F\" oldes$^*$ \and Mouhamadou Sy}

\address{Juraj F\" oldes
\newline \indent Department of Mathematics, University of Virginia \indent 
\newline \indent  322 Kerchof Hall, Charlottesville, VA 22904-4137,\indent }
\email{foldes@virginia.edu}

\address{Mouhamadou Sy
\newline \indent Department of Mathematics, University of Virginia \indent 
\newline \indent  303 Kerchof Hall, Charlottesville, VA 22904-4137,\indent }
\email{ms3wq@virginia.edu}

\thanks{ J. F\" oldes is partly supported by the National Sicence Foundation under the grant NSF-DMS-1816408. 
\\
\indent
$^*$ corresponding author}

\title[Invariant measures and GWP for SQG]{Invariant measures and global well posedness  for SQG equation}

\begin{abstract}
We construct an invariant measure $\mu$ for the Surface Quasi-Geostrophic (SQG) equation and show that almost all functions in the support of $\mu$ are initial conditions of  global, unique solutions of SQG, that depend continuously on the initial data.  In addition, we show that the support of $\mu$ is infinite dimensional, meaning that 
it is not locally a subset of any compact set with finite Hausdorff dimension. Also, there are global solutions that have arbitrarily large initial condition. 
The measures $\mu$ is obtained via fluctuation-dissipation method, that is, as a limit of invariant measures for 
stochastic SQG with a carefully chosen dissipation and random forcing.

\noindent
\textbf{Keywords}: Surface Quasi Geostrophic equation, invariant measures, fluctuation-dissipation method, global solutions,  stochastic partial differential equations

\noindent
\textbf{MSC}: 35Q35, 76D03, 35R60, 60H15, 37C40   
\end{abstract}

\maketitle

\section{Introduction}

The goal of the present manuscript is to construct an invariant measure $\mu$ for the  Surface Quasi-Geostrophic (SQG) equation
\begin{equation}\label{sqg}
\Trs_t + \bfU \cdot \nabla \Trs = 0
\end{equation}
and to prove $\mu$ almost sure global well posedness of \eqref{sqg}. 
First, we establish
the
existence of invariant measures $(\mu_\alpha)_{\alpha > 0}$ for 
the stochastic SQG 
\begin{equation}\label{sqgs}
d \Trs + \bfU \cdot \nabla \Trs dt = -\al \Delta^2 \Trs dt + \alpha \nabla(|\nabla \Trs|^2 \nabla \Trs) + \sqrt{\alpha} d\eta \qquad \textrm{on } \TT \times (0, \infty) \,,
\end{equation}
and then we construct $\mu$ as a limit of $\mu_\alpha$ as $\alpha \to 0^+$. 
In addition, we prove that all functions in the support of $\mu$ are initial conditions of global, regular solutions, and the support of $\mu$ is infinite dimensional. Before 
we precisely formulate our main results, let us fix the notation and provide a motivation for our study. 

To avoid unnecessary technicalities associated with boundary conditions, we work on two dimensional flat torus $\TT$, however most of our techniques could be applied 
to domains
with boundary. 
Unless indicated otherwise, we always assume that $\Trs$ has zero mean for all times, that is, 
\begin{equation}
\int_{\TT} \Trs(x, t) \, dx = 0 \qquad t \geq 0 \,.
\end{equation}
We assume that   $\Trs : \RR_+\times \TT \to \RR$ has sufficient regularity (as detailed below), and 
$\bfU=(-\partial_y,\partial_x)(-\Delta)^{-\frac{1}{2}}\Trs=R^\perp\Trs$ is the Riesz transform of $\Trs$, that is, 
\begin{equation}\label{rietra}
\bfU = F^{-1} \left(-i \frac{\mathbf{\xi}}{|\mathbf{\xi}|}F(\Trs)\right) \,,
\end{equation} 
where $F$ and $F^{-1}$ denote respectively Fourier and inverse Fourier transform. 
As usual, we work with cylindrical Weiner process defined on a filtered probability space 
$(\Omega, \mathcal{F}, \mathcal{F}_{t\geq 0}, \mathbb{P})$ and our 
 stochastic forcing has the form 
\begin{equation}\label{dfns}
\eta (t,x) = \sum_{j = 1}^\infty a_j e_j(x) W_j(t) \,, 
\end{equation}
where $e_j$ are  eigenfunctions of $-\Delta$ on the torus $\TT,$ ordered such that the corresponding eigenvalues $\lambda_j > 0$ form a non-decreasing sequence, 
and $(a_j)$ is a
sequence of real numbers such that 
\begin{equation}
A_0 = \sum_{j = 1}^\infty a_j^2 < \infty \,.
\end{equation}
Finally,  we complement \eqref{sqgs} with appropriate initial condition specified below. 

The SQG equation \eqref{sqg} appears as a  model for the temperature 
of stratified atmosphere on the rapidly rotating planet or as  a model of ocean dynamics on certain scales 	\cite{Blumen} 
(for derivation, applications to ocean and atmosphere dynamics, and more references see \cite{pedlosky1992geophysical} or a more recent survey \cite{Lapeyre}). From mathematical 
perspective, the SQG equation attracted a lot of attention due to many similarities with three dimensional Euler equation. Most nobably, 
the vector $\nabla^\perp \Trs$ satsifies an analogue of the Euler equation in the vorticity form. In particular, both equations contain vortex stretching term
and a divergence free drift term, however one is posed in 2D whereas the other one in 3D and the constitutive laws are different, 
see the seminal work by Constantin, Majda, and Tabak \cite{CMT94} for more discussion.  

Although the local existence and uniqueness of smooth solutions of \eqref{sqg} was already resolved
in \cite{CMT94}, despite many efforts, the global existence of solutions on a torus remains open. The blow-up scenario proposed in
 \cite{CastroCordoba} was ruled out by precise numerical simulations in \cite{OhkitaniYamada} and analytically in 
\cite{Cordoba97, Cordoba98, ConstantinNieSch}. Another mechanism of gradient blow up based on the propagation of small instabilities in 
thin filaments was proposed in \cite{scott_2011}. 
 We remark that blow-up was constructed in \cite{CastroCordoba} for infinite energy initial conditions on $\RR^2$. 

In \cite{KiselevNazarov}, Kiselev and Nazarov 
showed that there exists a solution with initial condition having arbitrarily small initial conditions that attains arbitrary big norms in finite time. Later, motivated by a construction for the Euler equation \cite{Zlatos15}, it was shown in \cite{he2019small} that there are solutions of \eqref{sqg} with $W^{2, \infty}$ norm 
growing exponentially (along sub-sequence) as time goes to infinity
\begin{equation}\label{kni}
\sup_{t \leq T} \|\nabla^2 \Trs(\cdot, t)\|_{L^\infty} \geq e^{\gamma T} \qquad \textrm{for some } \gamma > 0 \,.
\end{equation}
Also, very little is known about non-equilibrium global smooth solutions for SQG. In fact the only example was given in \cite{castro2016global}, where 
with a rigorous, computer assisted proof the authors proved a global existence for initial conditions on 
one dimensional bifurcation branch close to a specific radial equilibrium.  

We just briefly remark that one can also consider weak solutions of \eqref{sqg}, which are known to be global \cite{Marchand, Resnick}. However, the uniqueness
of weak solution was an challenging open problem \cite{DeLellisSzekelyhidi} that was solved by establishing non-uniqueness in \cite{BuckmasterShkollerVicol}, see 
also \cite{rusin2011,  AzzamBedrossian}. Also, several regularized models (e.g. additional dissipation, or smoother constitutive law) were
 introduced for which one can prove global well posedness of solutions, 
 see \cite{KiselevNazarovVolberg, CaffarelliVasseur, ConstantinVicol} and \cite{ConstantinIIyerWu}.

In the present manuscript we utilize fluctuation-dissipation method to construct global solutions of \eqref{sqg}. The idea is to add a regularizing higher order differential terms, which guarantee global well posedness, and a stochastic forcing that keeps the energy balance in \eqref{sqgs}. Note that the strength of the forcing and the coefficients 
of the smoothing operators are carefully balanced.  
We prove that the stochastic SQG equation possesses an invariant measure supported on appropriate Sobolev spaces and by passing to the limit, we obtain 
an invariant measure $\mu$ for the deterministic SQG \eqref{sqg}. Then, we investigate properties of $\mu$.  

Let us describe known results for problems close to \eqref{sqgs}. 
Well posedness of stochastic SQG with either additive or multiplicative noise and additional sub-critical smoothing (dissipation of the form $(- \Delta)^\beta$, 
$\beta > \frac{1}{2}$)
was studied in \cite{RocknerZhuZhu}. The authors proved that the problem is pathwise globally well posed and under additional assumptions they showed that there exists a unique invariant measure, which is ergodic, and attracts all distributions at an algebraic rate.  Later large deviation principles for stochastic SQG were proved in 
\cite{LiuRocknerZhu}. 
Note that stochastic Quasi-geostrophic (which contains additional Laplacian compared to SQG) was earlier studied in \cite{BrannanDuanWanner, HuangGuoHan}. A regularization of \eqref{sqg} with help of the random diffusion was proved in 
\cite{BuNaStWi} for sufficiently small smooth initial conditions. 

Below, we first investigate the pathwise global well posedness of \eqref{sqgs} and 
then we prove for each $\alpha > 0$  the existence of invariant measure $\mu_\alpha$ supported on $H^2$. The choice of bi-Laplacian in \eqref{sqg} rather than 
Laplacian, stems from the fact, that we need $\mu_\alpha$ to be supported on $H^2$
rather than $H^1$. Otherwise, 
after passing $\alpha \to 0$, 
we would obtain a measure supported on $H^1$ which is not sufficient for the proof of uniqueness of solutions of \eqref{sqg} (see below).
Our first main result is stated in the following theorem, for more precise formulation see Theorems \ref{TheoremGWP} and
\ref{TheoremStationaryMeasures} below. 

\begin{Thm}\label{thm-met}
Assuming $A_0 < \infty$ and appropriate moment bounds on the initial distribution (see Theorem \ref{TheoremGWP} below), 
almost surely there exists a pathwise global solution to \eqref{sqgs}. Furthermore, 
\eqref{sqgs} admits at least one stationary measure $\mu_\alpha$ supported on 
$H^2(\TT) \cap W^{1, 4}(\TT)$.  
\end{Thm}

The proof of global pathwise, well posedness follows from a standard framework -- Galerkin approximation and passage to the limit. Since we were not able to locate a suitable result 
in the literature,  we provide sketch of the proof with 
appropriate references. Then, the existence of the invariant measure follows from 
moment bounds on solutions and Kryloff-Bogoliouboff theorem \cite{KryloffBogoliouboff}, see also \cite{da2014stochastic}. Note that using 
standard coupling techniques, one can prove that $\mu_\alpha$ is in fact 
a unique invariant measure. For proofs in settings close to ours, we refer to 
\cite{BricmontKupiainenLefevere2001, DaPratoDebussche2003, 
FlandoliMaslowski1, FGHRT, FGHRW, 
HairerMattingly2008, KS12}

Before discussing 
 convergence properties of measures constructed in Theorem \ref{thm-met}, let us 
 summarize known results. 
Passing $\alpha \to 0$ and consecutive  analysis of limiting measure $\mu$ was done 
for Euler equation in \cite{kuk_eul_lim} and \cite{GHSV}, where
it was proved that $\mu$ is supported on 
$H^1 \cap L^\infty$. Moreover, it was proved that for any compact set $S$ with finite Hausdorff dimension, one has $\mu(S) = 0$, that is, $\mu$ is infinite dimensional. The 
crucial property that allowed to prove the infinite dimensionality was the existence 
of infinitely many conservation laws. Also, it was shown that the support of $\mu$ contains solutions with large energy. 
Analogous results were obtained for KdV, Benjamin-Ono,  
Klein-Gordon, and Schr\"odinger equation in  \cite{KS04,sybo,sykg} (see also references therein). 
It is important to notice that in all previous examples the proof of the invariance of the limiting measure $\mu$ was based on the 
well posedness of the underlying deterministic equation, which is not known for SQG equation. 
Observe that the proof of invariance 
 for Euler equation \cite{KS12} does not require global well posedness of the 
 deterministic equation, nevertheless the 2D Euler equation is significantly simpler than SQG (which resembles 3D Euler equation). 
The construction of global solutions for septic NLS \cite{sy2019sure} (not know to be globally well posed), utilizes a different strategy: 
the fluctuation-dissipation is used only 
for Galerking approximations and the main obstacle is passage to the  limit (based on an argument of Bourgain \cite{bourg94}).

On the other hand, different construction based on Gibbs measures was used to construct global solutions and invariant measures for 
various, possibly globally ill posed, Hamiltonian systems (see e.g \cite{bourg94,tzvNLS06,oh2010invariance} and references therein). 
 However, the authors of \cite{nahmod2018global} identified a serious obstruction that 
 prevent a `traditional way' (e.g. as for 2D Euler \cite{albeverio1990global}) of building a Gibbs measure based on the conservation of $L^2$ norm for the SQG equation. 
Indeed, for functions in the support of such measure,  the nonlinearity of SQG (one degree less regular than Euler)
cannot be defined in the sense of distributions.

The main novelty of the paper is the proof that the set of measures $(\mu_\alpha)$ from 
Theorem \ref{thm-met} has an accumulation point $\mu$ which is an invariant measure 
for \eqref{sqg}. The invariance is understood with respect to the dynamics
induced by the stochastic equation \eqref{sqgs} and a passage $\alpha \to 0$.  Furthermore, we prove that if the initial condition belongs to the support of $\mu$, 
the corresponding solution of \eqref{sqg} is global. Hence, it is important to estimate the size of the support of $\mu$. 
Althought the SQG is similar to 3D Euler equation, a notable difference is the existence of infinitely many conserved quantities, that allows us to 
prove that $\mu(K) = 0$ for any compact set $K$ with finite Hausdorff dimension. Also, we show that $\mu$ is not supported only on small functions, meaning that the support of $\mu$ must contain functions with arbitrarily large 
$H^{\frac{3}{2}}$ norm. The last statement follows from the fact, that we can choose 
the noise in the fluctuation dissipation method and obtain $\mu$ with large moments (see Corollary \ref{cor:ubb} below).

Before we proceed denote 
\begin{equation}\label{asn}
A_s := \sum_{j = 1}^\infty \lambda_j^{s} a_j^2 \,.
\end{equation}
and recall that regular solutions of SQG equation admit, the following set of conservation laws
\begin{align*}
E_{\frac{-1}{2}}(\Trs) &=\frac{1}{2}\int_{\TT}|(-\Delta)^\frac{-1}{4}\Trs|dx \,, \\
M(\Trs) &=\frac{1}{2}\int_{\TT} \Trs^2dx. 
\end{align*}

The next theorem contains our main results, for more general assertions see Theorems \ref{TheoremPassagetotheLimit}, \ref{theoremabsolutecontinuity},  and  \ref{theoremdimension} below.

\begin{Thm}\label{thm:main}
Assume $A_0 < \infty$. 
As $\alpha\to 0$, there is an accumulation point $\mu$ for the sequence $(\mu_\alpha)$ satisfying the following properties
\begin{enumerate}[(1)]
\item $\mu$ is a probability measure concentrated on the Sobolev space 
$H^2(\TT)$, that is,
\begin{equation}
\mu(H^2(\TT))=1.
\end{equation}
\item \label{item_flow} For $\mu$ almost all data $\Trs_0$, there is a unique function $\Trs \in C(\RR^+, H^1)\cap L^2_{loc}(\RR^+,H^2)$ satisfying the equation $\eqref{sqg}$ with $\Trs(0,x)=\Trs_0(x)$. Define a flow $\rho$ for $\eqref{sqg}$ as  $\rho_t(\Trs_0) = \Trs(\cdot, t; \Trs_0)$.  
\item The flow $\rho_t$ is continuous on $H^1$.
\item $\mu$ is invariant under $\rho_t$. 
\item $\mu$ satisfies the estimates
\begin{align}
\int_{L^2}\left(\|\Trs\|_{H^\frac{3}{2}}^2
-  \int_{\TT} | \nabla \Trs |^2 \nabla \Trs \cdot \nabla (-\Delta)^{-\frac{1}{2}} \Trs  dx 
\right) \mu(d\Trs) &=\frac{A_{\frac{-1}{2}}}{2} \,, \\
\int_{L^2} ( \|\Trs\|_{H^2}^2 + \|\Trs\|_{W^{1,4}}^4)\mu(d\Trs)<\infty. 
\end{align} 
\item \label{item_infdim}$\mu$ is infinite-dimensional  in the sense that it vanishes on finite-dimensional compact sets.
\item\label{item_absolutcont} The conservation laws of random variables  $M(\Trs)$ and $E_\frac{-1}{2}(\Trs)$  are absolutely continuous with respect to the Lebesgue measure on $\RR$.
\end{enumerate}
\end{Thm}

The regularity of functions in the support of $\mu$ (being $H^2$ by Theorem \ref{thm:main}, 1.) is a direct consequence of the support of $\mu_\alpha$, which follows from the 
regularizing term $\Delta^2$. Note that $L^2H^2$ is a minimal smoothness required to prove uniqueness of solutions, that is, well posedness claimed in Theorem \ref{thm:main} part 
(2). 
Replacing $\Delta^2$ by $\Delta^\beta$ with $\beta < 2$, yields the existence of $(\mu_\alpha)$ and $\mu$ supported on $H^\beta$, but solutions 
of \eqref{sqg} with initial conditions in the support of $\mu$ are too weak to establish uniqueness.  

On the other hand the choice of $\Delta^2$ instead of $\Delta$ brings several obstacles. For example, 
for the proof of part (6) in Theorem \ref{thm:main} we needed to
introduce additional smoothing term $\Delta_4 \Trs := \nabla(|\nabla \Trs|^2 \nabla \Trs)$ into 
\eqref{sqgs} ($\Delta_p \Trs := \nabla(|\nabla \Trs|^{p - 2} \nabla \Trs)$ is called 
$p$-Laplacian). The reason is that the expression $\langle \Delta^2 \Trs, f(\Trs) \rangle$ is neither positive nor bounded from below
for all $\Trs$, and for large set of functions $f$, a minimal requirement for the general framework, see details in Section \ref{section_QualProperties}.  
The addition  of $\Delta_4$ guarantees that $(\mu_\alpha)$, and consequently $\mu$ are supported also on $W^{1, 4}$ with forth order moment bounds. 
Then, we can bound $\langle \Delta^2  \Trs - \Delta_4 \Trs, f(\Trs) \rangle$ from below for any $f$ that has bounded derivatives up to fourth order, which suffices 
for our purposes. 

If one wishes to construct invariant measures for \eqref{sqg} on smoother spaces, for example $H^\beta$ for $\beta > 2$, and prove infinite dimensionality 
of such measure, then one has to correct the smoothing operator $\Delta^\beta$ by appropriate quasilinear positive definite operator such as $p$-Laplacian. 
Since the proofs are technically involved we decided not to present them here.  

 Concerning the support of $\mu$, by part \ref{item_infdim} of Theorem \ref{thm:main}, it cannot be contained in any compact set of finite Hausdorff dimension. Moreover, $\ref{item_absolutcont}$ implies, that the support of $\mu$ is not merely a countable union of level sets of the conservation laws. Also, the following corollary asserts that there are arbitrarily large initial data that give rise to 
 global solutions.

 \begin{Cor}\label{cor:ubb}
 Given any constant $K$, denote $S_K = \{\Trs:  \|\Trs\|_{W^{1, 4}}^4 + \|\Trs\|_{H^{\frac{3}{2}}}^2 \geq K\}$. Then there is $\Trs_0 \in S_K$ such that  the solution
 of \eqref{sqg} with $\Trs(0) = \Trs_0$ is global. More generally, there exists a sequence $(a_j)$ (see the definition of the noise \eqref{dfns}) such that for the measure $\mu$
 constructed in Theorem \ref{thm:main} one has $\mu(S_K) > 0$. 
 \end{Cor}

\begin{proof}[Proof of Corollary \ref{cor:ubb}]
Choose the sequence $(a_j)$ such that $A_{-\frac{1}{2}} = 4 CK$ and $A_0 < \infty$, where $C$ is a constant depending only on the size of $\TT$.
 If $\mu (S_K) = 0$,  then, by Theorem \ref{thm:main}, (5) one has by H\" older and Poincar\' e inequalities
\begin{align}
2CK = \frac{A_{-\frac{1}{2}}}{2} &\leq \int_{L^2}\left(  \|\Trs\|^2_{H^{\frac{3}{2}}} + \int_{\TT} |\nabla \Trs|^3 |\nabla (- \Delta)^{\frac{1}{2}} \Trs| dx \right) \mu(d\Trs )
\leq  \int_{L^2}  \|\Trs\|^2_{H^{\frac{3}{2}}} + \|\nabla \Trs\|_{W^{1, 4}}^3 \|\Trs\|_{L^4}  \mu(d\Trs )\\
&\leq C \int_{L^2}  \|\Trs\|^2_{H^{\frac{3}{2}}} + \|\nabla \Trs\|_{W^{1, 4}}^4  \mu(d\Trs ) = 
  C \int_{L^2 \setminus S_K}  \|\Trs\|^2_{H^{\frac{3}{2}}} + \|\nabla \Trs\|_{W^{1, 4}}^4  \mu(d\Trs ) \leq C K \,,
\end{align}
a contradiction. The first statement follows from the second one and Theorem \ref{thm:main} parts (1), (2).
\end{proof}

Another consequence of Theorem \ref{thm:main} follows from the Poincar\'e recurrence theorem.

\begin{Cor}\label{cor:prt}
For $\mu$-almost every $u_0\in H^2$, there is a sequence $(t_k)_k$ increasing to $\infty$  such that
\begin{equation}
\lim_{k\to\infty}\|\rho_{t_k}u_0-u_0\|_2=0.
\end{equation}
\end{Cor}

Corollary \eqref{cor:prt}
could be the reason why the estimate \eqref{kni} requires $\sum_{t \leq T}$, that is, the solution increases only along the sequence of times (solutions might return 
infinitely many times to a neighbourhood of 
the initial condition).

A natural question is  whether  the set of solutions  constructed in Theorem \ref{thm:main}  is a subset of equilibria to \eqref{sqg}. This seems to be a 
very non-trivial question to which we do not have a definitive answer. However, we have the following alternative:
\begin{itemize}
\item[(a)] The support of $\mu$ is not a subset of the equilibria of \eqref{sqg}, and then the flow $\rho_t$ constructed in Theorem 
\ref{thm:main}, part $\ref{item_flow}$ yields non-trivial global solutions. 
\item[(b)] The support of $\mu$ coincides with the equilibria of $\eqref{sqg}$, for 
any choice of (sufficiently regular) noise. In that case, we would have a remarkable stability property of the equilibria for both \eqref{sqg} and \eqref{sqgs} with small $\alpha$.\
\end{itemize} 
Recall that the linear stability of equilibria of \eqref{sqg} were studied in 
\cite{Friedlander2005}.

Let us remark that in the context of equations having only discrete set of equilibria, for instance the case of some power type nonlinearities, 
$\ref{item_infdim}$ and $\ref{item_absolutcont}$ of Theorem \ref{thm:main}  imply that   the alternative (a) above occurs. Also, we propose in the appendix a general example of 
a finite-dimensional system having continuous set of equilibria, but the support of the 
inviscid measure not being subset of equilibria. 

\smallskip

\paragraph{\bfseries Organization of the paper.}
In Section $\ref{section_GWP}$, we prove probabilistic global well-posedness for the stochastic equation $\eqref{sqgs}$. Moment bounds for such solutions are given in Section $\ref{section_ProbaEstimates}$ and based on moment bounds we construct  stationary measures for any $\alpha > 0$ in Section $\ref{section_StationMeasures}$. Section $\ref{section_InviscidLimit}$ contains principal results of the paper, and we prove there the existence of invariant measure for \eqref{sqg}, and global well posedness on its support. In 
Section $\ref{section_QualProperties}$, we combine the probabilistic estimates and Krylov lemma to establish qualitative properties (infinite dimensionality of the support). Finally, Appendix \ref{sec:ham} include details about the invariant measures for 
finite dimensional Hamiltonian systems and in Appendices \ref{sec:Ito} and \ref{sec:imbed} we respectively recall It\^ o formula in infinite dimensions and a
proof of a parabolic embedding.  
  
\subsection{General notations}  
The following notation is used throughout the paper.
\begin{itemize}[label = $\ast$]
\item $C^\infty_0(\RR)$ is the space of functions $f:\RR \to \RR$ that are infinitely differentiable and compactly supported.
\item For any $1 \leq p \leq \infty$ and $s \in \RR$, we denote $L^p(\TT)$ and $W^{s, p}(\TT)$ the usual Lebesgue respectively Sobolev spaces.  We also set 
$H^s (\TT) = W^{s, p}(\TT)$. Often for the clarity of presentation we do not indicate the domain $\TT$ and we write $L^p$, $W^{k, p}$, and $H^k$. 
\item When a fixed $T>0$ is clear from context, for any Banach space $X$ define the spaces $CX = C([0,T],X)$,  $L^p X=L^p([0,T],X)$, and $W^{s, p}X = W^{s, p} ([0,T],X)$. Sometime if needed, we 
indicate the variable of the space as a subscript, for example $C_t H^s_x$. The spaces are equipped with usual parabolic norms denoted for example $\|\cdot \|_{CX}$. 
\item We write $\|\cdot\|$ instead of $\|\cdot\|_{L^2(\TT)}$. 
\item $C_{\textrm{loc}}X$ denotes the space of functions that are locally continuous in time with values in $X$. Analogously we define $L^p_{\textrm{loc}} X$ and 
$W^{s, p}_{\textrm{loc}} X$.
\item We denote $\mathfrak{p}(X)$ the set of Borel probability measures on $X$. 
\item The non-decreasing sequence $(\lambda_m)_{m \geq 1}$ contains all eigenvalues of $-\Delta$ on $\TT$ and with corresponding normalized eigenvectors $(e_m)_{m \geq 1}$.
\item For a probability measure $\mu$ on $X$, we denote by $\E_{\mu}(f)$, the average of $f$ with respect to $\mu$: 
$$
\int_Xf(x)\mu(dx).
$$
\item The Riesz transform of $\Trs$ is given denoted $R^\perp\Trs=(-\partial_y,\partial_x)(-\Delta)^{-\frac{1}{2}}\Trs.$ Note that
Riesz transform satisfies (see \cite{stein1993harmonic}), for any $p\in[1,\infty)$
\begin{equation}\label{rtd}
\|R^\perp v\|_{L^p} \leq C \|v\|_{L^p} \,.
\end{equation}
\end{itemize}

\section{Global solutions for the stochastic SQG}\label{section_GWP}

In this section we establish the path-wise global well posedness of solutions of \eqref{sqgs}, and therefore 
prove the first part of Theorem \ref{thm-met}. Although the proof 
follows from a framework used several times in the literature, we were unable to locate the precise reference that 
would cover our situation. Rather than providing all details, we show how to satisfy assumptions of  
 \cite[Theorem 1.3]{Liu-Rockner} and explain how the proof in \cite{Liu-Rockner} needs to be modified. 

\begin{Thm}\label{TheoremGWP}
Fix any $\alpha > 0$, any $T > 0$, and any $p \geq 1$.  
Also, fix
any $\mathcal{F}_0$ measurable (see filtration for our Brownian motion random variable $\Trs_0$ 
 with $\E \|\Trs_0\|_{L^2}^p < \infty$, 
and any noise $\eta$ of the form \eqref{dfns} with $A_0 < \infty$. 
Then, there exists a unique adapted solution $\Trs$
of \eqref{sqgs} satisfying $\Trs(0) = \Trs_0$ and almost surely
\begin{equation}
\Trs \in C ([0, T], L^2(\TT))  \cap L^2 ([0, T], H^{2}(\TT)) 
\cap L^{4}([0, T], W^{1, 4}(\TT))  \,.
\end{equation}
Furthermore, 
\begin{equation}\label{ltn}
\E \sup_{t \in [0, T]}\|\Trs(t)\|^{2p} +
\alpha \E \int_0^T \|\Trs\|^{2p- 2} (\|\Trs\|^2_{H^2} + \|\Trs\|_{W^{1, 4}}^4) ds \leq 
C(T, \alpha, p, \|\Trs_0\|) \,.
\end{equation}
\end{Thm}

\begin{proof}
The proof closely follows the proof of \cite[Theorem 1.3]{Liu-Rockner}, see also 
\cite{Liu-Rockner2}. 
 However, since our differential 
operators have different scalings, we have to slightly modify the  arguments. 
We only highlight differences. 
For easier comparison, we use the notation form \cite{Liu-Rockner}.

Recall that $(e_n)_{n \geq 1}$ is an orthonormal basis of $L^2$ and denote 
$H_n = \textrm{span}\{e_1, \cdots, e_n\}$. Let 
$P_n : H^{-2} \to H_n$ be the orthogonal projection defined by 
\begin{equation}
P_n y = \sum_{i = 1}^n \langle y, e_i\rangle e_i, \qquad y \in H^{-2} \,.
\end{equation}
For each $n \geq 1$ consider the stochastic equation on $H_n$
\begin{equation}\label{fds}
dX^{(n)} = P_n (A X^{(n)}) dt + \sqrt{\alpha} d\eta_n \,, \qquad 
X_0^{(n)} = P_n \Trs_0 
\end{equation}
where 
\begin{equation}
AX = - \alpha \Delta^2 X + \alpha \nabla (|\nabla X|\nabla X) - Y \nabla X, 
\qquad Y = R^\perp X
\end{equation}
and 
\begin{equation}
\eta_n(x, t) = \sum_{j = 1}^n a_j e_j(x) W_j(t) \,.
\end{equation}
The existence and uniqueness of solutions of \eqref{fds} is classical and 
follows from  \cite[Section 1]{Krylov-Rozovskii}, see also \cite[Theorem 3.1.1]{Prevot-Rockner}. 

We have the following a priori estimates for $X^{(n)}$.

\begin{Lem}\label{l:coe}
For every $T > 0$ there exists $C_T$ depending on $A_0$, $p$ and $\alpha$, but
independent of $n$ such that for
each $n \geq 1$
\begin{equation}
\E \sup_{t \in [0, T]} \|X^{(n)}(t)\|_{L^2}^{2p} + 
 \int_0^t  \|X^{(n)}(t)\|_{L^2}^{2p-2}
( \|X^{(n)}\|_{H^2}^2 + \|X^{(n)}\|_{W^{1,4}}^4) ds \\
\leq  C_T (\E \|X^{(n)}(0)\|_{L^2}^{2p} + 1) \,.
\end{equation}
\end{Lem}

\begin{proof}[Proof of Lemma \ref{l:coe}]
The proof follows from \cite[Lemma 2.2]{Liu-Rockner}, see also proof of \eqref{cet}
below for the idea of the proof.
\end{proof}

Define the spaces
\begin{equation}
Y_1 :=  L^2([0, T] \times \Omega, H^2) \qquad 
Y_2 := L^4([0, T] \times \Omega, W^{1, 4})
\end{equation}
and
\begin{equation}
K = Y_1 \cap Y_2
\end{equation}
and the dual of $K$
\begin{equation}
K^* = Y_1^* + Y_2^*
\end{equation}
equipped with the usual intersection and sum norms
\begin{equation}\label{sxn}
\|X\|_{K} = \max\{\|X\|_{Y_1}, \|X\|_{Y_2}\}, \qquad 
\|X\|_{K^*} = \inf \{\|X_1\|_{Y_1} + \|X_2\|_{Y_2}, X_1 + X_2 = X\} \,.
\end{equation} 
Note that 
\begin{equation}
Y_1^* :=  L^2([0, T] \times \Omega, H^{-2}) \qquad 
Y_2^* := L^\frac{4}{3}([0, T] \times \Omega, W^{-1, \frac{4}{3}})
\end{equation}
Since 
\begin{equation}
\|A X^{(n)}\|_{K^*} \leq \alpha \|\Delta^2 X^{(n)}\|_{Y_1^*} +
 \alpha \|\nabla (|\nabla X^{(n)}|\nabla X^{(n)})\|_{Y^*_2} + \|Y^{(n)} \nabla X^{(n)}\|_{Y_1^*}
\end{equation}
and 
\begin{align*}
 \|\Delta^2 X^{(n)}\|_{Y_1^*} &= \E \|X^{(n)}\|_{L^2H^2}^2 \\
 \|\nabla (|\nabla X^{(n)}|^2\nabla X^{(n)})\|_{Y^*_2} &= \E \||\nabla X^{(n)}|^3\|_{L^{\frac{4}{3}}L^{\frac{4}{3}}} =  
 \E \|\nabla X^{(n)}\|_{L^{4}L^{4}}^3 \leq C \left(\E \|X^{(n)}\|_{L^{4}W^{1,4}}^4\right)^{\frac{3}{4}}\,, \\
 \|Y^{(n)} \nabla X^{(n)}\|_{Y_1^*}^2 &\leq C \E \||Y^{(n)}|X^{(n)}\|_{L^2L^2}^2  \leq 
C \E \|Y^{(n)}\|_{L^4L^4}^2 \|X^{(n)}\|_{L^4L^4}^2 \leq
   C \E \|X^{(n)}\|_{L^4W^{1,4}}^4\,.
\end{align*}
Thus, from Lemma \ref{l:coe} follows that 
\begin{equation}\label{kse}
\|A X^{(n)}\|_{K^*} \leq C
\end{equation}
with $C$ independent of $n$, because $\E M(X_0) \leq \E \|\Trs_0\|^2 < \infty$.

The continuity of the map (assumption (H1) in \cite{Liu-Rockner}) 
$s \mapsto \langle A(X_1 + s X_2), X\rangle$, $i = 1, 2$ is easy to verify for any $X_1, X_2, X \in H^2$. 
Also,  the local monotonicity assumption ((H2) in \cite{Liu-Rockner})
\begin{equation}\label{moo}
\langle A(X_1) - A(X_2), X_1 - X_2 \rangle \leq (K + \kappa(X_2)) \|X_1 - X_2\|^2_{L^2}
\end{equation}
is valid for our operator $A$. Indeed, note that $p$-Laplacian is monotone operator (see e.g.
\cite[Proposition 30.10]{Zeidler}), thus \eqref{moo} holds true for $AX  =  \nabla(|\nabla X|^2 \nabla X)$ with $f = \rho = \kappa \equiv 0$. Since $Y_1$ is divergence free,  \eqref{rtd}, $L^\infty \hookrightarrow H^2$, and H\" older and Young inequalities yield
\begin{align}
\langle A(X_1) - A(X_2), X_1 - X_2\rangle &\leq - \alpha\|X_1 - X_2\|^2_{H^2} +  
\langle Y_2\nabla X_2 -Y_1\nabla X_1, X_1 - X_2 \rangle \\
&= 
-\alpha\|X_1 - X_2\|^2_{H^2} +
\langle Y_2\nabla X_2 -Y_1\nabla X_2, X_1 - X_2 \rangle \\
&=  
-\alpha\|X_1 - X_2\|^2_{H^2} +
\langle (Y_2 - Y_1)\nabla X_2, X_1 - X_2 \rangle \\
&\leq -\alpha\|X_1 - X_2\|^2_{H^2} + 
\|Y_2 - Y_1\|_{L^2} \|\nabla X_2\|_{L^2} \|X_1 - X_2\|_{L^\infty}
\\
&\leq 
C_\alpha \|X_2 - X_1\|_{L^2}^2 \|\nabla X_2\|_{L^2}^2 \,,
\end{align}
where we used that
\begin{align}
Y_1\nabla X_1=Y_1\nabla (X_1-X_2)+Y_1\nabla X_2
\end{align}
and that, by an integration by part and the property $\nabla\cdot Y_2=0$,
\begin{align}
-\langle Y_1\nabla (X_1-X_2),X_1-X_2\rangle=\langle\nabla\cdot Y_1,\frac{1}{2}(X_1-X_2)^2\rangle =0.
\end{align}
Thus, assumption (H2) in \cite{Liu-Rockner} holds with $K = 0$
and $\rho(X) = \|X\|_{H^1}^2$. Next, the assumption (H3) in \cite{Liu-Rockner}:
\begin{equation}
2 \langle A(X), X \rangle  + \delta \|X\|_{H^2}^\gamma \leq  K\|X\|^2_{L^2} + f(t)
\end{equation}
holds in our case, due to cancellation in the non-linear term, 
for $\delta < 2\alpha$,  $f \equiv 0$, and $\gamma = 2$. 

Finally, as stated in \cite[Remark 3.2, (4)]{Liu-Rockner} the growth assumption in \cite[(H4)]{Liu-Rockner} is only needed to prove that 
$\E \|A(X^{(n)})\|_{K^*}^2$ is uniformly bounded in $n$, which was 
already established in \eqref{kse}.

The rest of the proof follows line by line the same as in \cite[proof of Theorem 1.1]{Liu-Rockner}. 
\end{proof}

\section{Probabilistic estimates for the stochastic flow}\label{section_ProbaEstimates}

In this section, we derive moment bounds 
on solutions of \eqref{sqgs}, that were constructed in Section \ref{section_GWP}. Our choice of norms is dictated by the conserved quantities of \eqref{sqg}, 
and it is essential to keep track of dependencies of constants on $\alpha$. The proofs are based on energy estimates and It\^ o lemma recalled in Appendix 
\ref{sec:Ito}. 
  
For $A_0$ defined in \eqref{asn}, observe that for any $p>0$ 
\begin{equation}
\int_{\TT} \left( \sum_{j = 1}^\infty  a_j^2 (e_j(x))^2  \right)^{\frac{p}{2}} \, dx \leq \frac{1}{(2\pi)^p}\int_{\TT} \left( \sum_{j = 1}^\infty  a_j^2   \right)^{\frac{p}{2}} \, dx = (2\pi)^{2-p} A_0^{\frac{p}{2}} < \infty 
\end{equation}
and recall the notation
\begin{align}
M(\Trs)=\frac{1}{2}\int_{\TT} \Trs^2dx \,.
\end{align}

\begin{Thm}
Assume $A_0 < \infty$. Then, the solution $\Trs$ constructed in Theorem $\ref{TheoremGWP}$ satisfies the following properties:

\begin{enumerate}
\item If $\E \|\Trs(0)\|_{H^{-1/2}}^2<\infty,$ then for any $t \geq 0$
\begin{multline}
\E \|\Trs(t)\|_{H^{-\frac{1}{2}}}^2 + 2\alpha \E \int_0^t  \| \Trs(s) \|_{H^{\frac{3}{2}}}^2  ds
+ 2 \alpha \E \int_0^t \int_{\TT} | \nabla \Trs(s) |^2 \nabla \Trs \cdot \nabla (-\Delta^{-\frac{1}{2}}) \Trs  dx ds
\\
= \E \|\Trs_0\|_{H^{-\frac{1}{2}}}^2  +\alpha A_{-1}t. \label{hmj}
\end{multline}
\item If $\E M^q (\Trs(0))<\infty$ for some $q\geq 1$,  then for any $t \geq 0$
\begin{multline}\label{cet}
\E M^q (\Trs(t)) + 2\alpha q \E \int_0^t  M^{q - 1} (\Trs)  ( \|\Trs\|_{H^2}^2 + \|\Trs\|_{W^{1,4}}^4) ds \\
=  \E M^q (\Trs(0)) + \alpha q \E \int_0^t A_0 M^{q - 1} (\Trs) ds + 
2(q - 1)  M^{q - 2}  (\Trs) \sum_{j = 1}^\infty  a_j^2 \left( \int_{\TT}   e_j \Trs dx \right)^2  ds \,.
\end{multline}
In particular, when $q=1$, we have for any $t \geq 0$ 
\begin{align}
\E\|\Trs(t)\|^2 + 2\al\int_0^t\E(\|\Trs(s)\|_{H^2}^2 + \|\Trs(s)\|_{W^{1,4}}^4)ds =\E\|\Trs_0\|^2 + \al A_0t.\label{alp}
\end{align}

\end{enumerate}
\end{Thm}
\begin{proof}
We proceed in the following steps.
\paragraph{\bfseries Proof of \eqref{hmj}.}

We use that $H^{-\frac{1}{2}}$ is conserved  for the SQG equation \eqref{sqg} 
(see \eqref{hoi} below). Clearly, the function $\Trs \mapsto \|\Trs\|_{H^{-1/2}}^2 = \|\Delta^{-1/4} \Trs \|^2$ satisfies assumptions \eqref{intro_KS_Ito_F1} and \eqref{intro_KS_Ito_F2} of Theorem \ref{intro_KS_Ito_theorem} with $s=-\frac{1}{2}$. To satisfy \eqref{intro_KS_Ito_F3}, notice that
\begin{align}
\sum_{j = 1}^\infty  a_j^2  \int_0^t \E \left( \int_{\TT} (-\Delta)^{-\frac{1}{4}} \Trs  (-\Delta)^{-\frac{1}{4}} e_j dx \right)^2 ds &= 
\sum_{j = 1}^\infty  a_j^2  \int_0^t \E \left( \int_{\TT}  \Trs  
(-\Delta)^{- \frac{1}{2}} e_j dx \right)^2 ds  \\
&= \sum_{j = 1}^\infty \frac{ a_j^2}{\lambda_j}  \int_0^t \E \left( \int_{\TT}  \Trs  e_j dx \right)^2 ds  \\
&\leq C  \sum_{j = 1}^\infty \frac{ a_j^2}{\lambda_j}  \int_0^t \E \int_{\TT}  \Trs^2   dx  ds \\
&= C A_{-1} \E  \int_0^t   \|\Trs\|^2   ds < \infty \,,
\end{align}
where the last inequality follows from \eqref{ltn}. Thus, Theorem \ref{intro_KS_Ito_theorem} yields
\begin{multline}
\E \|\Trs(t)\|_{H^{-1/2}}^2 = \E \|\Trs(0)\|_{H^{-1/2}}^2 + 
2\E \int_0^t \int_{\TT} (-\Delta)^{-\frac{1}{4}} \Trs (-\Delta)^{-\frac{1}{4}}(-\alpha \Delta^2 \Trs 
+ \al \nabla( |\nabla \Trs|^2 \nabla \Trs)
 - \bfU \nabla \Trs ) 
 \\
 + \alpha \sum_{j = 1}^\infty  a_j^2 ((-\Delta)^{-\frac{1}{4}} e_j)^2  \, dx  ds \,.
\end{multline}
Using that $\bfU = \nabla^{\perp}(-\Delta)^{\frac{1}{2}} \Trs$ is divergence free, we obtain 
\begin{equation}\label{hoi}
\begin{aligned}
 \int_{\TT} (-\Delta)^{-\frac{1}{4}} \Trs  (-\Delta)^{-\frac{1}{4}} (\bfU \nabla \Trs) \, dx &=
\int_{\TT} (-\Delta)^{-\frac{1}{2}} \Trs   \bfU \nabla \Trs \, dx \\
&=
 - \int_{\TT} ( \nabla (-\Delta)^{-\frac{1}{2}} \Trs \cdot \nabla^{\perp} (-\Delta)^{-\frac{1}{2}} \Trs) \Trs \, dx = 0 \,.
\end{aligned}
\end{equation}
Hence, using integration by parts and Fourier representation of fractional Laplacian we obtain
\begin{align}
\E \|\Trs(t)\|_{H^{-1/2}}^2 &= \E \|\Trs(0)\|_{H^{-1/2}}^2 +
2\alpha \E \int_0^t \int_{\TT} (-\Delta)^{-\frac{1}{4}} \Trs (-\Delta)^{-\frac{1}{4}}(- \Delta^2 \Trs + \nabla (|\nabla \Trs|^2\nabla \Trs)) \, dx ds  + 
\alpha \sum_{j = 1}^\infty \frac{ a_j^2}{\lambda_j}  t\\
&= \E \|\Trs(0)\|_{H^{-1/2}}^2 
- 2\alpha \E \int_0^t \|\Trs\|_{\frac{3}{2}}^2 ds -  
2\alpha \E \int_0^t \int_{\TT} (-\Delta)^{-\frac{1}{2}} \nabla \Trs  |\nabla \Trs|^2\nabla \Trs \, dx ds+ \alpha A_{-1}t 
\end{align}
and \eqref{hmj} follows.

\paragraph{Estimate for $M^q(\Trs).$}
Next, we turn our attention to the moment bounds for $M^q$ with $q \geq 1$. Clearly, the function $\theta \mapsto M^q (\Trs) = \frac{1}{2}\|\Trs\|^{2q}$ satisfies assumptions \eqref{intro_KS_Ito_F1} and \eqref{intro_KS_Ito_F2} of Theorem \eqref{intro_KS_Ito_theorem} with $s = 0$. In order to obtain 
 \eqref{intro_KS_Ito_F3}, we need to estimate  the quadratic variation of the martingale term. Since $(e_j)$ are bounded,  
\begin{align}
\sum_{j = 1}^\infty  a_j^2 \E \int_0^t\left( \|\Trs\|^{q - 1} \int_{\TT} \Trs e_j(x)  dx \right)^2 ds \leq C A_0 \E \int_0^t \|\Trs\|^{2q - 2} \int_{\TT} |\Trs|^2  dx ds =  
C A_0 \E \int_0^t  \|\Trs\|^{2q }   ds < \infty \,,
\end{align}
where the last inequality follows from \eqref{ltn}.
Hence, by Theorem \ref{intro_KS_Ito_theorem} we obtain for any $q \geq 1$
\begin{align}
\E M^q (\Trs(t)) &= \E M^q (\Trs(0)) + q \E \int_0^t M^{q - 1} (\Trs) \int_{\TT}  2\Trs   (-\alpha \Delta^2 \Trs  + \alpha \nabla (|\nabla \Trs|^2 \nabla \Trs) - \bfU \nabla \Trs
) +  \alpha \sum_{j = 1}^\infty  a_j^2 e_j^2  \, dx ds 
\\
&\qquad +  2\alpha(q - 1) M^{q - 2}  (\Trs) \sum_{j = 1}^\infty  a_j^2 \left( \int_{\TT}   e_j \Trs dx \right)^2  ds  \,.
\end{align}
An integration by parts and the property $\nabla \cdot \bfU = 0$ imply that $\int_{\TT} \Trs \bfU \nabla \Trs dx = 0$,  and  \eqref{cet} follows after integration by parts.
\end{proof}

\section{Stationary measures for the stochastic SQG}\label{section_StationMeasures}

In this section we construct invariant measures for the stochastic SQG equation \eqref{sqgs} and establish its moment bounds, which finishes the second part of 
Theorem \ref{thm-met}. As above, it is necessary 
to keep track of the parameter $\alpha$, since below we pass $\alpha$ to zero. Also note that the moment estimates are equalities, which will be important 
in the proof of non-degeneracy of the limiting measure. The proof of existence of invariant measures is based on the Kryloff-Bogoliouboff theorem, and the moment bounds
follow from bounds on solutions established in Section \ref{section_ProbaEstimates}.

\subsection{Construction and basic estimates}

\begin{Thm}\label{TheoremStationaryMeasures}
Assume $A_0 < \infty$. 
For any $\alpha \in (0,1),$ the equation $\eqref{sqgs}$ admits at least one stationary measure $\mu_\alpha$ supported on $L^2$ and satisfying the following properties:
\begin{align}
\int_{L^2}\left(\|\Trs\|_{H^\frac{3}{2}}^2
-  \int_{\TT} | \nabla \Trs |^2 \nabla \Trs \cdot \nabla (-\Delta)^{-\frac{1}{2}} \Trs  dx 
\right) \mu_\alpha(d\Trs) &=\frac{A_{\frac{-1}{2}}}{2}, \label{est_L2_mualpha}\\
\int_{L^2}(\|\Trs\|_{H^2}^2 + \|\Trs\|_{W^{1,4}}^4)\mu_\alpha(d\Trs) &=\frac{A_0}{2} \label{est_H1_mualpha}.
\end{align}
\item More generally, for any $q\geq 1,$
\begin{multline}
\int_{L^2}M^{q - 1} (\Trs)   (\|\Trs\|_{H^2}^2 + \|\Trs\|_{W^{1,4}}^4)\mu_\al(d\Trs) 
\\
=  \int_{L^2} \frac{A_0}{2} M^{q - 1} (\Trs) + (q - 1)  M^{q - 2}  (\Trs) \sum_{j = 1}^\infty  a_j^2 \left( \int_{\TT}   e_j \Trs dx \right)^2\mu_\alpha(d\Trs) \,.\label{est_powerq_mu_alpha}
\end{multline}
\item In particular, $\mu_\alpha(H^2)=1$,  any $\alpha > 0$, and any 
$q \geq 1$ there is $C$ independent of $\alpha$ such that 
\begin{equation}\label{igi}
\int_{L^2}M^{q - 1} (\Trs)  ( \|\Trs\|_{H^2}^2 + \|\Trs(s)\|_{W^{1,4}}^4) \mu_\al(d\Trs) \leq C \,.
\end{equation}
\end{Thm}

\begin{proof}
 \textbf{Tightness and existence of stationary measures.}
Let $\Trs_\alpha$ be the solution of \eqref{sqgs} with $\Trs_\alpha(0) = 0$ almost surely, that is $\Trs_\al(0)$ is distributed as  the Dirac measure concentrated at $0$.
Then, by \eqref{alp}, one has 
\begin{align}
 2\alpha \E \int_0^t \| \Trs \|_{H^2}^2 ds \leq  \alpha A_0 t\,,
\end{align}
and consequently
\begin{equation}\label{frs}
 \frac{1}{t} \E \int_0^t \| \Trs \|_{H^2}^2 ds \leq C  \,.
\end{equation}
 For each $t > 0$ define the Borel probability measure on $L^2(\TT)$ as
\begin{equation}
\mu_\al^t (A) = \frac{1}{t} \int_0^t \Prb(\Trs_\al(s) \in A) \, ds \,,
\end{equation}
where $ A$ is any Borel set in  $H^2(\TT) \,.$
Then, 
\begin{equation}\label{smt}
\int_{L^2(\TT)} \|\Trs\|^2_{H^2} \mu_\al^t(d\Trs) = \frac{1}{t} \E \int_0^t
 \| \Trs_\al(s) \|_{H^2}^2 ds \leq C. 
\end{equation}
In particular, if $B_R$ is a ball  in $H^2$ of radius $R$ centered at $0$ and $B^c_R = H^2 \setminus B_R$, then 
by Chebyshev inequality and \eqref{smt} one has
\begin{equation}
\mu_\al^t (B_R^c) =  \frac{1}{t} \int_0^t \Prb( \|\Trs_\al (s)\|_{H^2} \geq R) \, ds \leq  \frac{1}{t} \int_0^t \frac{ \E  \|\Trs(s)\|_{H^2}^2}{R^2} \, ds \leq \frac{C}{R^2} \,.
\end{equation}
Since $B_R$ is compact in $H^{2 - \delta}$, $\delta > 0$, then for each $\alpha > 0$,
the set of measures $(\mu_\al^t)_{t > 0}$ is tight, and therefore by Pokhorov theorem it is compact. 
For any sequence $(t_n)$ with $t_n \to \infty$, one has that 
$\mu_\al^{t_n}$ has a  weakly convergent subsequence converging to $\mu^*_\al$.  The Bogoliubov-Krylov argument (see e.g. \cite{da2014stochastic}) implies that 
$\mu^*_\al$ is stationary for \eqref{sqgs}.

Also, by using \eqref{cet} with $\Trs(0) = 0$, we obtain for any $q \geq 1$
\begin{align}
\int_{L^2}M^{q - 1} (\Trs) (  \|\Trs\|_{H^2}^2 &+ \|\Trs(s)\|_{W^{1,4}}^4)\mu_\al^t(d\Trs) 
\leq C \int_{L^2} M^{q - 1} (\Trs) \mu_\alpha^t(d\Trs)  \\
&=  C \int_{ \|\Trs\|_{H^2}^2 \geq R} M^{q - 1} (\Trs) \frac{ \|\Trs\|_{H^2}^2}{ \|\Trs\|_{H^2}^2} \mu_\alpha^t(d\Trs) + 
 C\int_{ \|\Trs\|_{H^2}^2 < R} M^{q - 1} (\Trs) \mu_\alpha^t(d\Trs)\\
 &\leq 
 \frac{C}{R} \int_{ \|\Trs\|_{H^2}^2 \geq R} M^{q - 1} (\Trs)  \|\Trs\|_{H^2}^2 \mu_\alpha^t(d\Trs) + 
 C\int_{ \|\Trs\|_{H^2}^2 < R} M^{q - 1} (\Trs) \mu_\alpha^t(d\Trs)
 \,.
\end{align}
Choosing $R = 2C$ and using $M(\Trs) \leq C  \|\Trs\|_{H^2}$ we have
\begin{equation}
\int_{L^2}M^{q - 1} (\Trs) (  \|\Trs\|_{H^2}^2 + \|\Trs(s)\|_{W^{1,4}}^4)
 \mu_\al^t(d\Trs) \leq C\int_{ \|\Trs\|_{H^2}^2 < 2C} M^{q - 1} (\Trs) \mu_\alpha^t(d\Trs) \leq C 
\end{equation}
were $C$ is independent of $t$. Then, by the Portmanteau theorem, the same 
inequality holds true with $\mu^t_\al$ replaced by $\mu_\al$ and \eqref{igi} follows.

\textbf{Estimates for the stationary measures.}
Denote by $\mu_\alpha$ any invariant measure constructed by the above procedure. 
Let $\Trs_0^\alpha$ be a random variable
with the law $\mu_\alpha$ and let $\Trs^\alpha$ be the solution of \eqref{sqgs} with initial condition $\Trs_0^\alpha$. Then, by \eqref{igi} one has 
$\E \|\Trs_0^\alpha\|^{q - 1} < \infty$ and $\E \|\Trs_0^\alpha\|_{H^{-\frac{1}{2}}}^2 < \infty$  for any $q \geq 1$. 

Also, the invariance of $\mu_\alpha$ implies 
$\E \|\Trs^\alpha(t)\|^2_X = \E \|\Trs^\alpha_0\|^2_X$ for $X$ being $L^2$, $W^{1, 4}$, or $H^2$, and consequently
by \eqref{alp}  
\begin{align}\label{ihj}
t \int_{L^2} \|\Trs\|_{H^2}^2 + \|\Trs\|_{W^{1,4}}^4\mu_\alpha(d\Trs) &= 
 \E \int_0^{t} \|\Trs^\alpha(s)\|^2_{H^2} + \|\Trs^\al(s)\|_{W^{1,4}}^4 ds 
  = t \frac{A_0}{2} 
\end{align}
and \eqref{est_H1_mualpha} follows. 

Similarly,  in \eqref{hmj} and \eqref{cet}, 
using $\E\|\Trs^\alpha(t)\|_{H^\frac{-1}{2}}^2=\E\|\Trs^\alpha_0\|_{H^\frac{-1}{2}}^2$ 
and  $\E M^q(\Trs(t)) = \E M^q(\Trs(0))$ (by the invariance of $\mu_\alpha$)
we obtain respectively \eqref{est_L2_mualpha} and  \eqref{est_powerq_mu_alpha}.
Recall that \eqref{igi} was already proved. 
\end{proof}

\section{Inviscid limit}\label{section_InviscidLimit}

This section contains proofs of essential assertions of the manuscript detailed in Theorem \ref{thm:main} parts 1--5. In particular, we prove prove convergence of 
measures $\mu_\alpha$ constructed in Theorem \ref{TheoremStationaryMeasures} to an invariant measure $\mu$ for the deterministic SQG equation. Furthermore we show that 
almost all points in the support of $\mu$ are initial conditions for regular global solutions.

\begin{Thm}\label{TheoremPassagetotheLimit}
If $A_0 < \infty$, 
there exists a measure $\mu$ supported on $H^2$ with the following properties: 
\begin{enumerate}
\item For almost every $\Trs_0 \in \textrm{supp}(\mu)$, there exists a unique, global (existing for all positive times) solution $\Trs$ of \eqref{sqg} with  $\Trs\in C_tH^1_x\cap L^2_tH^2$. Furthermore, for any $t \geq 0$, the mas $\Trs_0 \mapsto \Trs(t, \Trs_0)$ is a continuous on $H^1$. 
\item The measure $\mu$ is invariant for \eqref{sqg}, meaning that for every Borel set $A$ in $H^2$, one has 
$\mu\{\Trs_0 : \Trs(t, \Trs_0) \in A\} = \mu(A)$.
\item For any $q \geq 1$ we have the moment bounds
\begin{align}
\int_{L^2}\left(\|\Trs\|_{H^\frac{3}{2}}^2
-  \int_{\TT} | \nabla \Trs |^2 \nabla \Trs \cdot \nabla (-\Delta)^{-\frac{1}{2}} \Trs  dx 
\right) \mu(d\Trs) &=\frac{A_{\frac{-1}{2}}}{2}\,, \label{est_L2_lim_meas}\\
\int_{L^2} M^{q - 1}(\Trs)( \|\Trs\|_{H^2}^2 + \|\Trs\|_{W^{1,4}}^4)\mu(d\Trs)&\leq C. \label{est_H1_lim_meas}
\end{align} 
\end{enumerate}
\end{Thm}

Before providing details of the proof, let us first sketch the general strategy. 

\begin{Rmk}
Compared to the known results we face different challenges since it is not known whether the equation \eqref{sqg} is globally well posed. This poses several challenges. 
After verifying the tightness of measures $(\mu_\alpha)$ and passing $\mu_\alpha \to \mu$ as $\alpha \to 0$, we obtain moment bounds for $\mu$, however we cannot 
immediately conclude that almost all functions in the support are initial conditions of global solutions. This problem is not solved even if we prove that $\mu$ is invariant. For example, there can be a set $M_t$ of measure zero that contain functions that cease to exist at time $t$. Since $(M_t)_{t > 0}$ form an uncountable family, 
we cannot conclude that the union $\cup M_t$  has zero measure.  
\newline
For that reason we use the ``lifted" measures 
$\nu_\alpha$ supported on solutions of \eqref{sqgs} rather than on initial conditions. 
To pass $\alpha \to 0$ and conclude that the limiting measure $\nu$ is supported
on solutions of \eqref{sqg}, we have to obtain compactness (tightness) of $(\nu_\alpha)$ in spaces of time dependent functions. 
This follows from improved temporal bounds for the solutions of \eqref{sqgs}. Also, these bounds imply that the restriction of the measure $\nu$ at the initial time is $\mu$. 
\newline
Using the Skorokhod theorem we find stationary random variables $\Trs_\alpha$
distributed as $\nu_\alpha$ that converge almost surely to $\Trs$, which solves 
\eqref{sqg}. In addition, $\Trs(0)$ is distributed as $\mu$.  
\newline
To prove the uniqueness of $t \mapsto \Trs(t)$ we crucially use that the operator in the fluctuation-dissipation method is bi-Laplacian instead of Laplacian, and therefore 
$\Trs$ is supported on $L^2H^2$, a regularity space sufficient to guarantee uniqueness and continuous dependence on initial conditions. 
\end{Rmk}

\begin{proof}{Proof of Theorem \ref{TheoremPassagetotheLimit}}
The proof is divided into several parts. Proof of 1. follows from Proposition \ref{p:fpt}, 
part 2 follows from Lemma \ref{l:pjc} and the proof of 3 follows from Proposition \ref{p:mbm}. 
\end{proof}

If $\Trs_\al$ is a solution of $\eqref{sqgs}$ with $\Trs_\al(0)$ distributed as 
$\mu_\al$ (see Theorem \ref{TheoremStationaryMeasures}), then due to the invariance, $\Trs_\al(t)$ is
distributed as of $\mu_\al$ for any $t \geq 0$. 
We can either view $t \mapsto \theta_\al (t)$  as a random process with range in a space of $x$ dependent functions or alternatively, we can view $\theta_\al$ as a random variable on a space of $(x, t)$ dependent functions in
 $L^2_{loc}(\RR_+,H^2)$ (see \eqref{smt}). 

Denote $\nu_\al$ the distribution of $\Trs_\al$ and by the invariance of $\mu_\al$
one has $\Prb(\Trs_\al(t_0) \in A) =  \mu_\al (A)$ for any $t_0 \geq 0$, and any 
Borel set $A$ in  $H^{-\delta}$,  $\delta \in [0, 1)$. Observe that $\mu_\alpha$
is supported on $H^2$ and we can trivially (by zero) extend it to the larger space $H^{-\delta}$. 
Hence, 
\begin{equation}
\int \chi_{A \times \{t_0\}} d \nu_\alpha = 
\nu_\al (A \times \{t_0\}) = \Prb(\Trs_\al(t_0) \in A) =  \mu_\al (A) 
= \int_A \chi_{A} d\mu_\al 
\,,
\end{equation}
where $\chi_Z$ denotes the characteristic function of a set $Z$.  The 
linearity of integrals and the dominated convergence theorem also implies 
\begin{equation}\label{bci}
\int_{C_tH^{-\delta}_x} g(\Trs(t_0)) d\nu_\al (\Trs) =  
\int_{H^{-\delta}_x} g(\Trs)  d \mu_\al (\Trs)
\end{equation}
for any bounded continuous function $g: H^{-\delta}_x \to \RR$.

 Fix $T > 0$ and  define $I = (0, T) \subset \RR$.

\begin{Rmk}
In this section we implicitly assume that all spaces are defined on the time interval 
$I$. For example, $L^2H^2 = L^2(I, H^2(\TT))$ or $H^1_tL^4_x = H^1(I, L^4(\TT))$. 
We use the notation, say $L^4$ (single space), to denote $L^4(\TT)$, that is, 
we do not specify regularity in time. 
\newline
Also, $\Trs_\alpha$ denotes the solution of \eqref{sqgs}, that is, a function depending on $x$ and $t$, whereas $\Trs$ denotes the integration variable, that is, a function
depending on $x$ only.  
\end{Rmk}

By \eqref{igi}, and the invariance of $\mu_\alpha$ for any $q \geq 0$
\begin{align}\label{hdet}
C &\geq \int_0^T\int_{L^2_x} \|\Trs\|^q \|\Trs\|_{H^2}^2 d\mu_\al(\Trs) ds = 
\int_0^T \E \|\Trs_\al(s)\|^q \|\Trs_\al(s)\|_{H^2}^2 ds = 
\E \|\Trs_\al \|^q \|\Trs_\al\|_{L^2H^2}^2 \\
&= 
\int_{L^2_{t,x}} \|\Trs\|^q \|\Trs\|_{L^2H^2}^2 d\nu_\al(\Trs) \,,
\end{align}
where here and below $C$ is allowed to implicitly depend on $T$. To gain the temporal compactness in time, we prove the regularity of $\Trs_\alpha$ in time. 

\begin{Lem}\label{raam}
Set $\mathcal{X} = H^1L^2+H^1H^{-2}+ W^{1, \frac{4}{3}}W^{-1, \frac{4}{3}}+W^{\kappa,4}L^2$ with $\kappa\in(1/4,1/2)$ equipped with the standard sum 
norm (cf. \eqref{sxn}). Then
\begin{equation}
C \geq \E \|\Trs_\al\|_{\mathcal{X}}^\frac{4}{3} = \int_{L^2_{t,x}} \|\Trs\|_{\mathcal{X}}^\frac{4}{3} d\nu_\al(\Trs) \,,
\end{equation}
\end{Lem} 

\begin{proof}
Observe that 
\begin{align}
\Trs_\al(t)=\underbrace{\Trs_\alpha(0)-\int_0^t\bfU\cdot\nabla\Trs_\al ds}_{I}-
\al \underbrace{\int_0^t\Delta^2\Trs_\al ds}_{II} + \alpha \underbrace{\int_0^t \nabla(|\nabla \Trs_\al|^2\nabla \Trs_\al) ds}_{III}+
\sqrt{\al} \underbrace{\int_0^t d\eta}_{IV}. \label{sqgsinteg}
\end{align}
First, by 
 interpolation, \eqref{rtd}, and embeddings we have 
(spatial norm)
\begin{equation}\label{PassLimEstNonL}
\|\bfU\cdot\nabla\Trs\|_{L^2_{x}} \leq C\|\bfU\|_{L^4} \|\nabla \Trs\|_{L^4} \leq C \|\Trs\|_{H^{\frac{1}{2}}}  \|\Trs\|_{H^{\frac{3}{2}}}
\leq C \|\Trs\| \|\Trs\|_{H^2} \,,
\end{equation}
and by \eqref{est_H1_mualpha} and \eqref{igi} with $q=3$
\begin{align*}
\E \|I\|_{H^1_tL^2_x}^2\leq 2\left(\E\|\Trs_\al(0)\|_{L^2_x}^2+\E\|\bfU\cdot\nabla\Trs_\alpha\|_{L^2_{t,x}}^2\right)\leq C \,.
\end{align*}
Second, using \eqref{est_H1_mualpha}, we have
\begin{align}\label{ond}
\E\|II\|_{H^1_tH^{-2}_x}^2 \leq C \E\int_0^T\| \Delta^2\Trs\|_{H^{-2}}^2 dt=\E\int_0^T\|\Trs\|_{H^2}^2\leq CT.
\end{align}
Moreover,  
\begin{equation}
\|\nabla(|\nabla \Trs|^2\nabla \Trs)\|_{W^{-1,\frac{4}{3}}} \leq C \||\nabla \Trs|^3\|_{L^{\frac{4}{3}}} = C \|\nabla \Trs\|_{L^{4}}^3 = C \|\Trs\|_{W^{1,4}}^3 \,,
\end{equation}
and therefore, by \eqref{igi}
\begin{align}\label{don}
\E \|III\|^{\frac{4}{3}}_{W^{1,\frac{4}{3}}_tW^{-1,\frac{4}{3}}_x}
&\leq
\E\|\nabla(|\nabla \Trs_\al |^2\nabla \Trs_\al )\|_{L^{\frac{4}{3}}W^{-1,\frac{4}{3}}}^\frac{4}{3}\leq 
C\E  \int_0^T \|\Trs_\al\|_{W^{1,4}}^4 \leq C \,.
\end{align}
Finally, since for any $m$ and $0 \leq s \leq t$, 
$W_m(t) - W_m(s) \sim \mathcal{N}(0,t - s)$, we have 
$\E |W_m(t) - W_m(s)|^2= t - s$ and 
$\E|W_m(t) - W_m(s)|^4=3(t - s)^2$. By the independence of $W_j$ and $W_k$
for $j \neq k$, one has  
\begin{align}
\E \|IV(t)-IV(s)\|^4 &= 
\E \left(\sum_{j = 1}^\infty a_j^2   |W_j(t) - W_j(s)|^2  \right)^2 =
\E \sum_{j, k = 1}^\infty a_j^2 a_k^2   |W_j(t) - W_j(s)|^2 |W_k(t) - W_k(s)|^2  \\
&= 3 |t - s|^2  \sum_{j, k = 1}^\infty a_j^2 a_k^2  (1 + 2\delta_{jk}) \leq 
9|t - s|^2 \sum_{j, k = 1}^\infty a_j^2 a_k^2 \leq 9|t - s|^2 A_0^2 \,,
\end{align}
where $\delta_{ij} = 0$ if $i \neq j$ and $\delta_{ii} = 1$.  
 Consequently, if $\kappa < \frac{1}{2}$
\begin{align*}
\E\|IV\|_{W^{\kappa,4}L^2}^4&=\E\int_0^T\|IV\|_{L^2_x}^4dt+\E\int_0^T\int_0^T\frac{\|IV(t)-IV(s)\|^4}{|t-s|^{1+4\kappa}}dtds \\
&\leq C \int_0^T t^2 dt + C \int_0^T\int_0^T |t-s|^{1-4\kappa} dtds
\leq C(T).
\end{align*}
Overall, 
\begin{align}
\E \|\Trs_\al\|_\chi^{\frac{4}{3}} \leq 
\E \|I\|_{H^1L^2}^2 + \alpha\|II\|_{H^1H^{-2}}^2 +  
\alpha\|III\|_{W^{1, \frac{4}{3}}W^{-1, \frac{4}{3}}}^\frac{4}{3} + 
\E \|IV\|_{W^{\kappa,4}L^2}^4 \leq C,
\end{align}
where $C$ is independent of $\al \in (0, 1)$ and the result follows. 
\end{proof}

\begin{Prop}\label{prop:com}
For any $\delta>0$ denote $\mathcal{Y}_\delta = 
L^2 H^{2-\delta}\cap C H^{-\delta}$.
Let $\mathcal{X}$ be as in Lemma \ref{raam} for some $\kappa \in (\frac{1}{4}, \frac{1}{2})$. Then, for any $q \geq 0$
there is a constant $C$ independent of $\al$ such that 
\begin{align}\label{dvot}
\int \|\Trs\|^\frac{4}{3}_{\mathcal{X}} d\nu_\alpha 
&= \E\|\Trs_\al\|_{\mathcal{X}}^\frac{4}{3} \leq C \,,  \\ \label{nsi}
\int \|\Trs\|^q \|\Trs\|^2_{L^2H^2} d\nu_\alpha 
&= \E \|\Trs_\al\|^q \|\Trs_\al\|_{L^2H^2}^2 \leq C.
\end{align}
Moreover, for any $\delta > \frac{1}{3}$ the set of measures $(\nu_\al)_{\al}$ is tight in $\mathcal{Y}_\delta$.  Consequently, there is a sequence 
$(\nu_k):=(\nu_{\al_k})$ with $\al_k\to 0$ as $k\to\infty$, and a measure $\nu$ supported on $\mathcal{Y}_\delta$ such that $\nu_k$ converges weakly to $\nu$ as $k\to\infty$.
\end{Prop}

\begin{proof}
The estimates \eqref{dvot} and \eqref{nsi} follow from  Lemma \ref{raam}
and \eqref{hdet} respectively. 

We claim that $\mathcal{Y}_\delta$ is compactly embedded in $\mathcal{X} \cap L^2H^2$ for any $\delta > \frac{1}{3}$. 
Indeed, 
by \cite[Theorem $5.1$ and $5.2$]{lions69}, for any $\delta > 0$,
$L^2 H^2 \cap \mathcal{X}$ is compactly embedded in $L^2 H^{2-\delta}$. 
Also, for any $\delta > 0$,   \cite[Theorem 3.1]{LM}, \cite[Lemma $II.2.4$]{krylov95}, and standard Sobolev  embedding imply  that 
$L^2 H^2 \cap H^1H^{-2}$ and $L^2 H^2 \cap W^{\kappa,4} L^2$ are compactly embedded in $C H^{-\delta}$. 
Finally, by Appendix \ref{sec:imbed}, for any $\delta > \frac{1}{3}$,  $L^2 H^2 \cap W^{1, \frac{4}{3}}  W^{-1, \frac{4}{3}}$ is compactly embedded in $C H^{-\delta}$, 
and the claim follows. 

Let $B_R$ be the ball in $\mathcal{X}\cap L^2H^2$ of radius $R$ centered at the origin. By the just proved compactness, $B_R$ is compact in $\mathcal{Y}_\delta$. Furthermore, by Chebyshev inequality
\begin{equation}\label{tomn}
\nu_\alpha (B_R^c) = \Prb (\|\Trs_\al\|_{\mathcal{X}\cap L^2H^2} \geq R) \leq \frac{\E \|\Trs_\al\|_{\mathcal{X}\cap L^2H^2}^\frac{4}{3}}{R^\frac{4}{3}} \leq \frac{C}{R^\frac{4}{3}}
\end{equation} 
and therefore the set of measures $(\nu_\alpha)_\al$ is tight in $\mathcal{Y}_\delta$. The existence
 of appropriate sequence follows from Prokhorov theorem. 
\end{proof}

\begin{Lem}\label{l:sko}
Let $\nu_k  \to \nu$ in be as in Proposition \ref{prop:com}. Then, 
there is  a probability space $(\tilde{\Omega},\tilde{\mathbb{P}})$, on which is defined a sequence of random variables $(\tilde{\Trs}_k)$ and a random variable $\tilde{\Trs}$ having the following properties:
\begin{enumerate}
\item[(1)] The law of $\tilde{\Trs}$ is $\nu$ and for every $k$, the law of $\tilde{\Trs}_k$ is $\nu_k$.
\item[(2)] For any $\delta > \frac{1}{3}$, the sequence $\tilde{\Trs}_k$ converges to $\tilde{\Trs}$ almost surely, that is, for $\tilde{\mathbb{P}}$ almost every $\omega \in \tilde{\Omega}$
one has $\|\tilde{\Trs}_k(\omega) - \tilde{\Trs}(\omega)\|_{\mathcal{Y}_\delta} \to 0$ as $k \to \infty$.
\item[(3)] For each $k$, $\tilde{\Trs}_k$ satisfies \eqref{sqgsinteg}. 
\end{enumerate}
Furthermore, 
by passing to a sub-sequence if necessary, $\tilde{\Trs}_k$ converges weakly to $\tilde{\Trs}$ in 
$L^\frac{4}{3}(\tilde{\Omega}, \mathcal{X}) \cap L^2(\tilde{\Omega}, L^2H^2)$ and 
for any $q \geq 0$
\begin{align} \label{rbi}
\int \|\Trs\|_{\mathcal{X}}^{\frac{4}{3}} d\nu (\Trs) &= 
\E\|\tilde{\Trs}\|_{\mathcal{X}}^{\frac{4}{3}} \leq \liminf_{k \to \infty} \E\|\tilde{\Trs}_k\|_{\mathcal{X}}^{\frac{4}{3}} \leq  C_T, \\ \label{cjv}
\int \|\Trs\|^q \|\Trs\|_{L^2H^2}^{2} d\nu (\Trs) &= 
\E\|\tilde{\Trs}\|^q \|\tilde{\Trs}\|_{L^2H^2}^{2} \leq 
\liminf_{k \to \infty} \E\|\tilde{\Trs}_k\|^q \|\tilde{\Trs}_k\|_{L^2H^2}^{\frac{4}{3}} \leq  C_T.
\end{align}
\end{Lem}

\begin{proof}
Since $\mathcal{Y}_\delta$ with $\delta \in (\frac{1}{3}, \frac{1}{2})$ is a separable metric space, Skorokhod theorem (see  \cite[Theorem $11.7.2$]{dudley}) implies (1) and (2). Moreover, (3) follows analogously as in \cite[Section 4.3.4]{Bensoussan1995}. 

By Proposition \ref{prop:com}, $(\tilde{\Trs}_k)$ is uniformly bounded in 
$\mathcal{Z} = L^\frac{4}{3}(\tilde{\Omega}, \mathcal{X}\cap L^2H^2)$, and therefore, 
up to a subsequence, $(\tilde{\Trs}_k)$ weakly converges in $\mathcal{Z}$ 
to some $\hat{\Trs}$. Due to compactness of the embedding $\mathcal{X}\cap L^2H^2 \hookrightarrow \mathcal{Y_\delta}$ and since weak convergence implies
convergence almost surely (up to subsequence), one has that $(\tilde{\Trs}_k)$ 
converges almost surely in $\mathcal{Y}_\delta$ to $\hat{\Trs}$. Uniqueness of
the limit implies that $\tilde{\Trs} = \hat{\Trs}$, and in particular $(\tilde{\Trs}_k)$ weakly converges in $\mathcal{Z}$ 
to some $\hat{\Trs}$.

Finally, \eqref{rbi} and \eqref{cjv} follows from the weak lower semi-continuity of norms and \eqref{dvot}, \eqref{nsi} respectively.
\end{proof}

Next, we prove that  $\tilde{\Trs}$ satisfies \eqref{sqg} almost surely. Before proceeding, we prove the following auxiliary result.

\begin{Lem}\label{lcob}
Fix $\delta \in (\frac{1}{3}, \frac{2}{3})$ and recall $\mathcal{Y}_\delta = 
C_tH^{-\delta} \cap L^{2}_tH^{2 - \delta}$. 
For any sufficiently smooth $\Trs \in \mathcal{Y}_\delta$ one has
\begin{align}\label{PassLimIneqNonlin}
\|R^\perp\Trs\cdot\nabla\Trs\|_{L^2_{t}H^{-1}}\leq 
C\|\Trs\|_{\mathcal{Y}_\delta}^2.
\end{align}
Also, the map $B: \mathcal{Y}_\delta \to L^2_tH^{-1}_x$ defined as $B(\Trs) = R^\perp\Trs\cdot\nabla\Trs$
is continuous.
\end{Lem}

\begin{proof}
It suffices to prove the assertion for smooth functions and then use a standard 
argument to pass to the limit. 

For any smooth $\Trs_1, \Trs_2$, with help of \eqref{rtd}, Agmon's inequality, and interpolation, one has 
\begin{align}
\|B(\Trs_1) - B(\Trs_2)\|_{H^{-1}_x} &\leq 
\|R^\perp(\Trs_1 - \Trs_2)\cdot\nabla\Trs_1\|_{H^{-1}_x} + 
\|R^\perp \Trs_2 \cdot\nabla(\Trs_1 - \Trs_2)\|_{H^{-1}_x}  \\
&\leq 
\||R^\perp(\Trs_1 - \Trs_2)| \Trs_1\| + 
\||R^\perp \Trs_2| (\Trs_1 - \Trs_2)\| 
\\
&\leq
C(\|R^\perp(\Trs_1 - \Trs_2)\|  \|\Trs_1\|_{L^\infty_x} + 
\|R^\perp \Trs_2\|_{L^\infty_x} \|(\Trs_1 - \Trs_2)\| )
\\
&\leq
C \|\Trs_1 - \Trs_2\| ( \|\Trs_1\|_{H^{-\delta}_x}^{\frac{1}{2} - \frac{\delta}{4}}
\|\Trs_1\|_{H^{2-\delta}_x}^{\frac{1}{2} + \frac{\delta}{4}} +  
 \|R^\perp \Trs_2\|_{H^{-\delta}_x}^{\frac{1}{2} - \frac{\delta}{4}}
\|R^\perp \Trs_2\|_{H^{2-\delta}_x}^{\frac{1}{2} + \frac{\delta}{4}} )
\\
&\leq
C \|\Trs_1 - \Trs_2\|_{H^{-\delta}_x}^{1 - \frac{\delta}{2}}
\|\Trs_1 - \Trs_2\|_{H^{2 - \delta}_x}^{\frac{\delta}{2}} 
( \|\Trs_1\|_{H^{-\delta}_x}^{\frac{1}{2} - \frac{\delta}{4}}
\|\Trs_1\|_{H^{2-\delta}_x}^{\frac{1}{2} + \frac{\delta}{4}} +  
 \|\Trs_2\|_{H^{-\delta}_x}^{\frac{1}{2} - \frac{\delta}{4}}
\|\Trs_2\|_{H^{2-\delta}_x}^{\frac{1}{2} + \frac{\delta}{4}} )
\,.
\end{align}
To prove \eqref{PassLimIneqNonlin} we set $\Trs_2 \equiv 0$, $\Trs_1 = \Trs$. 
After integration in time and an application of Jensen's inequality we obtain
for $\delta \in (0, \frac{2}{3})$ 
\begin{align}
\int_0^T \|B(\Trs)\|_{H^{-1}_x}^2ds \leq C 
\int_0^T  
 \|\Trs\|_{H^{-\delta}_x}^{3 - \frac{3\delta}{2}}
\|\Trs\|_{H^{2-\delta}_x}^{1 + \frac{3\delta}{2}}   
 ds
 \leq 
 C_T \|\Trs\|_{L^{\infty}_tH^{-\delta}_x}^{3 - \frac{3\delta}{2}}  
 \|\Trs\|_{L^2_tH^{2-\delta}_x}^{1 + \frac{3\delta}{2}} \leq 
 C_T \|\Trs\|_{\mathcal{Y}_\delta}^4 \,.
\end{align}
To prove continuity of $B$ we observe that for $\delta \in (0, \frac{2}{3})$
\begin{align}
\int_0^T \|\Trs_1 - \Trs_2\|_{H^{-\delta}_x}^{2 - \delta}
\|\Trs_1 - \Trs_2\|_{H^{2 - \delta}_x}^{\delta} 
 &\|\Trs_1\|_{H^{-\delta}_x}^{1 - \frac{\delta}{2}}
\|\Trs_1\|_{H^{2-\delta}_x}^{1 + \frac{\delta}{2}} ds \\
&\leq 
\|\Trs_1 - \Trs_2\|_{L^\infty_tH^{-\delta}_x}^{2 - \delta} 
 \|\Trs_1\|_{L^\infty_tH^{-\delta}_x}^{1 - \frac{\delta}{2}}
\int_0^T 
\|\Trs_1 - \Trs_2\|_{H^{2 - \delta}_x}^{\delta} 
\|\Trs_1\|_{H^{2-\delta}_x}^{1 + \frac{\delta}{2}} ds \\
&\leq C_T \|\Trs_1 - \Trs_2\|_{L^\infty_tH^{-\delta}_x}^{2 - \frac{\delta}{2}} 
 \|\Trs_1\|_{L^\infty_tH^{-\delta}_x}^{1 - \frac{\delta}{2}}
\|\Trs_1 - \Trs_2\|_{L^2_tH^{2 - \delta}_x}^{\delta} 
\|\Trs_1\|_{L^2_tH^{2-\delta}_x}^{1 + \frac{\delta}{2}} \\
&\leq C_T
\|\Trs_1 - \Trs_2\|_{\mathcal{Y}_\delta}^{2}
 \|\Trs_1\|_{\mathcal{Y}_\delta}^{2} \,.
\end{align}
Thus, 
\begin{equation}
\|B(\Trs_1) - B(\Trs_2)\|_{L^2_tH^{-1}_x} \leq C_T 
\|\Trs_1 - \Trs_2\|_{\mathcal{Y}_\delta}(
 \|\Trs_1\|_{\mathcal{Y}_\delta} +  \|\Trs_1\|_{\mathcal{Y}_\delta})
\end{equation}
as desired. 
\end{proof}

\begin{Lem}\label{l:eer}
For $\tilde{\Trs}$ defined in Lemma \ref{l:sko} one has almost surely $\tilde{\Trs} \in C_tH^1_x\cap L^2_tH^2_x \cap H^1_tL^2_x$ locally in time. Furthermore, $\tilde{\Trs}$  almost surely satisfies 
\begin{equation}
\tilde{\Trs}(t)=\tilde{\Trs}(0)-\int_0^t\bfU\cdot\nabla\tilde{\Trs} ds 
\qquad t \geq 0\,,
\end{equation}
that is, $\tilde{\Trs}$ is a strong solution of \eqref{sqg} on $[0, T]$. 
\end{Lem}

\begin{proof}
Let $(\al_k)$ be as in Proposition \ref{prop:com}. 
Replacing $\al$ by $\al_k$ in \eqref{sqgsinteg} and setting 
$\tilde{\Trs}_k=\tilde{\Trs}_{\al_k}$ (see property Lemma \ref{l:sko} part (3)), 
we have
\begin{align*}
\tilde{\Trs}_k(t)=\tilde{\Trs}_k(0)-\int_0^t\bfU_k\cdot\nabla\tilde{\Trs}_k ds + \al_k\int_0^t
\nabla(|\nabla \tilde{\Trs}_k|^2 \nabla \tilde{\Trs}_k) - 
\Delta^2\tilde{\Trs}_k ds+\sqrt{\al_k}\zeta =I+\al_k II+\sqrt{\al_k}III \,.
\end{align*}
By Proposition \ref{prop:com} (cf. \eqref{ond} and \eqref{don})
\begin{align}\label{cvjd}
\al_kII, \ \sqrt{\al_k}III\to 0 
\end{align}
where the convergence is in $L^{\frac{4}{3}}(\tilde{\Omega}, \mathcal{X})$.
Then, by Chebyshev inequality, \eqref{cvjd} holds in probability and by passing to a sub-sequence, we can assume that \eqref{cvjd} holds 
almost surely in $\mathcal{X}$. Since $\mathcal{X} \hookrightarrow C_tH^{-2}_x$, 
\eqref{cvjd} holds almost surely in $C_tH^{-2}_x$.

If $\delta \in (\frac{1}{3}, \frac{2}{3})$, then  
Lemma \ref{lcob} yields  that
 $\tilde{\Trs} \mapsto R^\perp\tilde{\Trs}\cdot\nabla\tilde{\Trs}$ 
is continuous as map from $\mathcal{Y}_\delta$ into $L^2(I,H^{-1}_x)$. 
Hence,  as $k\to\infty$
\begin{align*}
\int_0^t\bfU_k\cdot\nabla\tilde{\Trs}_kds\to \int_0^t\bfU\cdot\nabla\tilde{\Trs} ds\quad \textrm{in }\ H^1_t H^{-1}_x  \qquad \textrm{a.s.}
\end{align*}
From Lemma \ref{l:sko}, part (2) follows almost surely 
\begin{align*}
\tilde{\Trs}_k\to \tilde{\Trs},\qquad \textrm{in } C_tH^{-\delta}_x.
\end{align*} 
Overall, almost surely we have for  any $t \in I$ 
\begin{equation} \label{thsq}
\tilde{\Trs}(t)=\tilde{\Trs}(0)-\int_0^t\bfU\cdot\nabla\tilde{\Trs} ds \quad \textrm{in } \ C_tH^{-2}_x.
\end{equation}
To obtain the regularity of $\tilde{\Trs}$, observe that \eqref{cjv} implies 
almost surely $\tilde{\Trs} \in L^2H^2$. 
Also, by interpolation, properties of Riesz transform,  and Agmon's inequality
\begin{equation}
\|\bfU\cdot\nabla\tilde{\Trs}\|_{L^2_{x}} \leq 
\|\bfU\|_{L^\infty_x} \|\nabla \tilde{\Trs}\|_{L^2_x} \leq C
\|\bfU\|_{L^2_x}^{\frac{1}{2}}\|\bfU\|_{H^2_x}^{\frac{1}{2}}
\|\Trs\|_{L^2_x}^{\frac{1}{2}}\|\Trs\|_{H^2_x}^{\frac{1}{2}} 
\leq C  \|\Trs\|_{L^2_x}\|\Trs\|_{H^2_x} \,.
\end{equation}
Consequently, by \eqref{cjv}
\begin{equation}\label{eont}
\E \|\bfU\cdot\nabla\tilde{\Trs}\|_{L^2_{x, t}}^2 
\leq C \E  \|\Trs\|_{L^2_{t,x}}^2 \|\Trs\|_{L^2_tH^2_x}^2 \leq C \,,
\end{equation}
and therefore almost surely $\partial_t \tilde{\Trs} \in L^2_{t,x}$. Then, 
the Lions-Magenes lemma (see  \cite[Theorem 3.1]{LM}) yields that $\tilde{\Trs}$ belongs a.s. locally to $C_tH^1_x\cap L^2_tH^2_x \cap H^1_tL^2_x$. 
\end{proof}

The proved regularity is exactly a borderline case for the proof of uniqueness. As such we cannot use direct energy estimates, but we have to employ more
subtle argument of Judovich, who used it for Euler equation, see \cite{Judovic1963, Majda2002}.  In particular, we need a precise estimates on the Sobolev embedding constants.

\begin{Lem}\label{luts}
Solution of \eqref{sqg} with $\Trs(0) = \Trs_0 \in H^1$ that belongs to $C_tH^1\cap L^2_tH^2 \cap H^1_tL^2_x$ is unique.
\end{Lem}

\begin{proof}
Let $\Trs_i \in C_tH^1\cap L^2_tH^2 \cap H^1_tL^2_x$,  $i=1,2$ be two solution of \eqref{sqg} with 
$\Trs_1(0)=\Trs_2(0)$. Then, $w=\Trs_1-\Trs_2$ satisfies
\begin{align}\label{dewf}
w_t=-R^\perp w\cdot\nabla \Trs_1-R^\perp\Trs_1\cdot\nabla w.
\end{align}
Testing with $w$ and using \eqref{rtd} yield 
\begin{align*}
\frac{d}{dt}\|w\|^2\leq2 |(w,R^\perp w\cdot\nabla \Trs_1)|\leq C\|w\|\|w\|_{L^{2p}}\|\nabla \Trs_1\|_{L^{\frac{2p}{p-1}}}. 
\end{align*}
By interpolation, we have for any $p \in (1, 2)$
\begin{align}
\|w\|_{L^{2p}}\leq \|w\|^{2-p}\|w\|_{L^{\frac{2p}{p-1}}}^{p-1}.\label{Lp-interp}
\end{align}
Using H\" older inequality, $p < 2$, and Sobolev inequality with precise constant (see e.g. \cite[Remark 1.5]{Chang2004}), we obtain 
\begin{align*}
\frac{d}{dt}\|w\|^2 &\leq C\|w\|^{3-p}\|w\|_{L^{\frac{2p}{p-1}}}^{p-1}\|\nabla \Trs_1\|_{L^{\frac{2p}{p-1}}} \leq 
C \sqrt{\frac{2p}{p - 1}} \|w\|^{2\delta}(\|\Trs_1\|_{H^2}^2+\|\Trs_2\|_{H^2}^2)^\frac{p}{2} 
\\
&\leq \frac{C}{\sqrt{1 - \delta}}\|w\|^{2\delta}(1+\|\Trs_1\|_{H^2}^2+\|\Trs_2\|_{H^2}^2) \,,
\end{align*}
where $\delta=\frac{3-p}{2} < 1$ and $p \in (1, 2)$.
Then,  after recalling that $w(0) = 0$ we have 
\begin{align*}
\frac{1}{1 - \delta} \|w(t)\|^{2(1 - \delta)} \leq \frac{C}{\sqrt{1-\delta}} \int_0^t\left(1+\|\Trs_1\|_{H^2}^2+\|\Trs_2\|_{H^2}^2\right) ds \,,
\end{align*}
and consequently
\begin{equation}
 \|w(t)\|^{2} \leq C\left( \sqrt{1-\delta} \int_0^t
 \left(1+\|\Trs_1\|_{H^2}^2+\|\Trs_2\|_{H^2}^2\right) ds \right)^{\frac{1}{1 - \delta}} \,.
\end{equation}
Since $\Trs_i \in L^2_tH^2$, then for any $t_0 \geq 0$ one has
$\sqrt{1 - \delta} \displaystyle{\int_0^{t_0}} (1+\|\Trs_1\|_{H^2}^2+
\|\Trs_2\|_{H^2}^2)ds<1$
for any $\delta < 1$ sufficiently close to $1$. Passing $p\to 1$ (or equivalently 
$\delta\to 1$), we arrive at
\begin{align*}
\|w(t)\|=0 \qquad \textrm{for any } t \leq t_0.
\end{align*}
Since $t_0$ was arbitrary, $\|w(t)\|=0$ for any $t \geq 0$, as desired.
\end{proof}

\begin{Lem}\label{l:pjc}
The law of $\tilde{\Trs} (t)$ is independent of 
$t$ and is equal to $\mu$. Here,  $\mu$ is a weak limit of (a sub-sequence)
 $(\mu_\alpha)$ as $\alpha \to 0$ in the space $H^{2 - \gamma}$, $\gamma > 0$, where $\mu_\alpha$ was defined in 
Theorem \ref{TheoremStationaryMeasures}.
Furthermore, $\mu$ is concentrated on $H^2$. 
\end{Lem}

\begin{proof}
From Chebyshev inequality and \eqref{igi} follows
\begin{equation}\label{tomn2}
\mu_\alpha (B_R^c) \leq \frac{1}{R^2} \int_{L^2} \|\Trs\|^2_{H^2} d \mu_\alpha(\Trs) 
\leq \frac{C}{R^2} \,,
\end{equation} 
where $B_R$ is a ball of radius $R$ in $H^2$ and $C$ is independent of $\alpha$. Since $H^2$ is compactly embedded in $H^{2 - \gamma}$, $\gamma > 0$, the Prokhorov theorem implies that there exists a weakly convergent sequence 
$(\mu_{\alpha_k})$ in $H^{2 - \gamma}$ to $\mu$. To prove that $\mu$ is supported 
on $H^2$ note that by \eqref{tomn2}
\begin{equation}\label{pmt}
\mu_\alpha (B_R) \geq 1 - \frac{C}{R^2}
\end{equation}
and by Portmanteau theorem, \eqref{pmt} holds with $\mu_\alpha$ replaced by $\mu$. Passing $R \to \infty$, one obtain $\mu(H^2) = 1$.

Fix $\tau \in [0, T)$ and a bounded continuous function $g : H^{-\delta} \to \RR$ and define $G(\Trs) = g(\Trs(\tau))$. 
We claim that for any $\delta \in (\frac{1}{3}, \frac{1}{2})$, $G : \mathcal{Y}_{\delta} \to \RR$ is bounded continuous. Indeed, if 
$\|\Trs_1 - \Trs_2\|_{\mathcal{Y}_\delta} < \varepsilon$, then 
$\|\Trs_1(\tau) - \Trs_2(\tau)\|_{H^{-\delta}} < \varepsilon$ and
\begin{equation}
|G(\Trs_1) - G(\Trs_2)| = |g(\Trs_1(\tau)) - g(\Trs_2(\tau))| \,,
\end{equation} 
and the boundedness and continuity of $G$ follows from the boundedness and continuity of $g$.

By Proposition \ref{prop:com} 
\begin{equation}\label{awl}
\lim_{k \to \infty} \int_{\mathcal{Y}_\delta} G(\Trs) d\nu_k(\Trs) =
\int_{\mathcal{Y}_\delta} G(\Trs) d\nu(\Trs)
\end{equation}
and by using \eqref{bci}, weak converges of $(\nu_k)$ and $(\mu_k)$ one obtains 
\begin{equation}
\E g(\tilde{\Trs}(\tau)) = 
\int_{\mathcal{Y}_\delta} g(\Trs(\tau)) d\nu(\Trs) =  
\lim_{k \to \infty} \int_{\mathcal{Y}_\delta} g(\Trs(\tau)) d\nu_k(\Trs)
= \lim_{k \to \infty} \int_{H^{-\delta}} g(\Trs) d\mu_k(\Trs) = \int_{H^{-\delta}} 
g(\Trs) d\mu(\Trs) \,.
\end{equation}
Thus for any $\tau$ the law of $\tilde{\Trs}$ is $\mu$ as desired. 
\end{proof}

\begin{Prop}\label{p:fpt}
For $\tilde{\Trs}$ defined in Lemma \ref{l:sko} one has almost surely 
$\tilde{\Trs} \in C(\RR_+, H^1_x)\cap L^2(\RR_+, H^2_x) \cap H^1(\RR_+,  L^2_x)$
and $\tilde{\Trs}$ satisfies \eqref{sqg}. Furthermore, $t \mapsto \tilde{\Trs}(t)$
is unique, and $\tilde{\Trs}$ depends continuously on initial conditions, that is, 
\begin{equation}\label{cost}
\lim_{\|\Trs_1(T_0)-\Trs_2(T_0)\|_{H^1_x}\to 0}
\sup_{t\in [T_0,T_1]}\|\Trs_1(t)-\Trs_2(t)\|_{H^1_x}=0 \,.
\end{equation}
We remark that by changing $t$ to $-t$, we can define solutions for all times, positive
or negative. 
\end{Prop}

\begin{proof}
By Lemmas \ref{l:eer} and \ref{luts}, for each integer $N > 0$ there exists almost surely a unique strong solution $\tilde{\Trs}^N$ of \eqref{sqg} on interval $[0, N)$. 
Since \eqref{sqg} is a deterministic equation, by almost surely we mean for $\mu$ almost every initial condition $\tilde{\Trs}(0)$ (see Lemma \ref{l:pjc}). Thus for each
integer $N > 0$ there exists a set $\mathcal{M}_N$, with $\mu(\mathcal{M}_N) = 1$
such that for each $\tilde{\Trs}_0 \in \mathcal{M}_N$, there exists a unique
solution of \eqref{sqg} on $[0, N)$ with $\tilde{\Trs}(0) = \tilde{\Trs}_0$. 

If we denote $\mathcal{M} = \cap_N \mathcal{M}_N$, then $\mu(\mathcal{M}) = 1$
and for each $\tilde{\Trs}_0 \in \mathcal{M}$,  and each $N$, 
there exists a unique solution of \eqref{sqg} on $[0, N)$ (see Lemma \ref{luts}), 
and global existence follows.  

By  slightly modifying the argument of
Lemma \ref{luts}, we could  prove \eqref{cost} with $H^1$ replaced by $L^2$.   
However, we need to modify the argument to obtain continuity with respect to the $H^1$ topology.

Test \eqref{dewf} by $\Delta w$, and 
let us first focus on the right hand side (using summation convention)
\begin{align*}
(\partial_{ii}^2w,\bfU^j\partial_j\Trs_1)+(\partial_{ii}^2w,\bfU^j_1\partial_jw)=(1)+(2)\,,
\end{align*}
where $\bfU = \mathcal{R}^\perp(w)$ and $\bfU_1 = \mathcal{R}^\perp(\Trs_1)$.
To estimate (2), we use an integration by parts and $\nabla\cdot\bfU=0$ to obtain
\begin{align}
(\partial_{ii}^2w,\bfU^j_1\partial_jw) &= - (\partial_{i}w, \partial_i\bfU^j_1\partial_jw) - (\partial_{i}w,\bfU^j_1 \partial^2_{ij}w)
= - (\partial_{i}w, \partial_i\bfU^j_1\partial_jw) - \frac{1}{2} (\partial_j(\partial_{i}w)^2,\bfU^j_1) \\
&= - (\partial_{i}w, \partial_i\bfU^j_1\partial_jw) \,,
\end{align}
and consequently  for any $p \in (1, 2)$
\begin{align}\label{faon}
  |(\partial_i w,\partial_jw\partial_{i}\bfU^j_1)|\leq \|w\|_{H^1}\|D \bfU_1\|_{L^{\frac{2p}{p-1}}}\|Dw\|_{L^{2p}}.
\end{align}
Using \eqref{Lp-interp}, \eqref{rtd},  and precise constant of embedding  as in the proof of Lemma 
\ref{luts}, one has
\begin{align}\label{saon}
|(2)| =
 |(\partial_i w,\partial_jw\partial_{i}\bfU^j_1)| \leq C\|w\|_{H^2}^{3-p}\|Dw\|_{L^{\frac{2p}{p-1}}}^{p-1}\|D\bfU_1\|_{L^{\frac{2p}{p-1}}}
\leq  \frac{C}{(p - 1)^\frac{p}{2}} \|w\|_{H^1}^{3-p} \left( \|\Trs_1\|_{H^2} + 
\|\Trs_2\|_{H^2}\right)^p \,.
\end{align}
On the other hand,
\begin{align*}
|(1)|&\leq |(\partial_i w,\partial_i\bfU^j\partial_j\Trs_1)| + |(\partial_i w,\bfU^j\partial_{ij}^2\Trs_1)|= |(3)| + |(4)|.
\end{align*}
As in \eqref{faon} and \eqref{saon} we obtain
\begin{equation}
|(3)| = |(\partial_i w,\partial_i\bfU^j\partial_j\Trs_1)| \leq  \frac{C}{(p - 1)^\frac{p}{2}} \|w\|_{H^1}^{3-p} \left( \|\Trs_1\|_{H^2} + \|\Trs_2\|_{H^2}\right)^p.
\end{equation}
 To estimate $(4)$,  the embedding $L^{q}\hookrightarrow H^1$ 
 (with precise constant \cite[Remark 1.5]{Chang2004}), \eqref{rtd} and  \eqref{Lp-interp}   imply
\begin{align}
|(4)| = |(\partial_i w,\bfU^j\partial_{ij}^2\Trs_1)| &\leq \|\Trs_1\|_{H^2}\|D w\|_{L^{2p}}\|\bfU\|_{L^\frac{2p}{p-1}}\leq 
 \|\Trs_1\|_{H^2}\|D w\|_{L^{2p}}\|w\|_{L^\frac{2p}{p-1}}
 \\
 &\leq 
C\|\Trs_1\|_{H^2} \|w\|_{H^1}^{2-p} \|Dw\|_{L^{\frac{2p}{p-1}}}^{p-1}\|w\|_{L^\frac{2p}{p-1}}  \leq
 \frac{C}{(p - 1)^\frac{p}{2}} \|w\|_{H^1}^{3-p} \left( \|\Trs_1\|_{H^2} + \|\Trs_2\|_{H^2}\right)^p \,,
\end{align}
where $C$ is independent of $p$. 
Combining all the estimates, and using that $p < 2$, we have  
\begin{align*}
\frac{d}{dt}\|w(t)\|_{H^1}^2\leq \frac{C}{(p - 1)^\frac{p}{2}}  \|w(t)\|_{H^1}^{2\delta}
(1+\|\Trs_1\|^2_{H^2} + \|\Trs_2\|^2_{H^2}),
\end{align*}
where $\delta=\frac{3-p}{2}$ and $p\in (1, 2)$. Thus, for any $T_0 < T_1$
\begin{align*}
\sup_{t\in [T_0,T_1]}\|w(t)\|_{H^1}^{2}\leq  \left(\|w(T_0)\|^{2(1-\delta)}_{H^1} + C(1 - \delta)^{\delta - \frac{1}{2}} \int_{T_0}^{T_1}(1+\|\Trs_1\|^2_{H^2}
+\|\Trs_2\|_{H^2}^2)ds\right)^{\frac{1}{1-\delta}}. 
\end{align*}
Passing $\|w(T_0)\|_1 \to 0$ implies for any $\delta \in (0 ,1)$
\begin{align*}
\lim_{\|w(T_0)\|\to 0}\sup_{t\in [T_0,T_1]}\|w(t)\|_1^{2}\leq  \left(C(1 - \delta)^{\delta - \frac{1}{2}} \int_{T_0}^{T_1}(1+\|\Trs_1\|^2_2
+\|\Trs_2\|_2^2)ds\right)^{\frac{1}{1-\delta}}. 
\end{align*}
 and finally letting $\delta  \to 1^+$, or equivalently $p \to 1^+$, we arrive at
\begin{align*}
\lim_{\|w(T_0)\|_1\to 0}\sup_{t\in [T_0,T_1]}\|w(t)\|_1=0,
\end{align*}
as desired.
\end{proof}

\begin{Prop}\label{p:mbm}
Under the assumption of Theorem \ref{TheoremPassagetotheLimit}, the 
relations  \eqref{est_L2_lim_meas} and \eqref{est_H1_lim_meas}.
\end{Prop}

\begin{proof}
Recall that by Lemma \ref{l:pjc}
$\mu_k \to \mu$ as measures on $H^{2-\gamma}$, $\gamma > 0$ and $\mu$ is 
supported on $H^2$. 

The inequality \eqref{est_H1_lim_meas} follows from \eqref{igi} and Portmanteau theorem, since $C$ in \eqref{igi} is independent of $\alpha$. 

To establish \eqref{est_L2_lim_meas}, frix $R \geq 1$ and  
let  $\psi_R : \RR \to [0, 1]$ be a $C^\infty$ cut off function with  $\psi_R(r) = 1$ for $|r| \leq R$ and $\psi_R(r) = 0$ for $|r| \geq R + 1$ . Denote $B_R$  the ball in $L^2$ centred at $0$ with radius $R$, and $B_R^c$ the complement of $B_R$ in $L^2$ and define
\begin{equation}\label{dith}
I(\Trs) = \|\Trs\|_{\frac{3}{2}}^2
-  \int_{\TT} | \nabla \Trs |^2 \nabla \Trs \cdot \nabla (-\Delta)^{-\frac{1}{2}} \Trs  dx 
 \,.
\end{equation}
Then, by \eqref{est_L2_mualpha} 
\begin{align*}
\frac{A_{\frac{-1}{2}}}{2} - \int_{B_R^c}  |I(\Trs)| \mu_k(d\Trs)\leq \int_{L^2}\psi_R(I(\Trs)) I(\Trs) \mu_k(d\Trs)\leq \frac{A_{\frac{-1}{2}}}{2}
+ \int_{B_R^c}  |I(\Trs)| \mu_k(d\Trs)\,.
\end{align*} 
Also, H\"older's inequality, interpolation,  \eqref{tomn2}, and \eqref{igi} imply
\begin{align*}
\int_{B_R^c}\|\Trs\|_{\frac{3}{2}}^2\mu_k(d\Trs)\leq 
\left(\int_{L^2}\|\Trs\|_\frac{3}{2}^\frac{8}{3}\mu_k(d\Trs)\right)^\frac{3}{4}(\mu_k(B_R^c))^\frac{1}{4}
\leq \left(\int_{L^2}\|\Trs\|^\frac{2}{3}\|\Trs\|_2^2 \mu_k(d\Trs)\right)^\frac{3}{4}(\mu_k(B_R^c))^\frac{1}{4}\leq \frac{C}{R^\frac{1}{2}}
\end{align*}
and by H\" older's and Gagliardo-Nirenberg inequalities
\begin{align}\label{fgnh}
\left|  \int_{\TT} | \nabla \Trs |^2 \nabla \Trs \cdot \nabla (-\Delta)^{-\frac{1}{2}} \Trs  dx  \right| &\leq  
\|| \nabla \Trs |^2 \nabla \Trs\|_{L^{\frac{4}{3}}} \|\nabla (-\Delta)^{-\frac{1}{2}}\Trs \|_{L^4} 
\leq C \|\nabla \Trs \|_{L^4}^3  \|\Trs \|_{L^4} \\
&\leq C \|\nabla \Trs \|_{L^4}^{\frac{10}{3}}  \|\Trs \|_{L^2}^{\frac{2}{3}} \,.
\end{align}
Hence, by the  \eqref{igi} and \eqref{tomn2}
\begin{align}
\int_{B_R^c}\left|  \int_{\TT} | \nabla \Trs |^2 \nabla \Trs \cdot \nabla (-\Delta)^{-\frac{1}{2}} \Trs  dx  \right| \mu_k(d\Trs)
&\leq C\int_{B_R^c}  \|\nabla \Trs \|_{L^4}^{\frac{10}{3}}  \|\Trs \|_{L^2}^{\frac{2}{3}}  \mu_k(d\Trs) \\
&\leq C \left(\int_{L^2}  \|\nabla \Trs \|_{L^4}^{4}  \|\Trs \|_{L^2}^{\frac{4}{5}}  \mu_k(d\Trs)\right)^{\frac{5}{6}} (\mu_k(B_R^c))^{\frac{1}{6}}
\leq \frac{C}{R^{\frac{1}{3}}} \,.
\end{align}
Thus,
\begin{align}\label{pmcv}
\frac{A_{-\frac{1}{2}}}{2} - C \left( \frac{1}{R^\frac{1}{2}} + \frac{1}{R^{\frac{1}{3}}}\right)
\leq \int_{L^2}\psi_R(I(\Trs)) I(\Trs) \mu_k(d\Trs)\leq \frac{A_{-\frac{1}{2}}}{2} + C \left( \frac{1}{R^\frac{1}{2}} + \frac{1}{R^{\frac{1}{3}}}\right).
\end{align} 
Furthermore, H\" older inequality and Sobolev embedding imply
\begin{align}
\left|  \int_{\TT} | \nabla \Trs |^2 \nabla \Trs \cdot \nabla (-\Delta)^{-\frac{1}{2}} \Trs  dx  \right| &\leq  
\|| \nabla \Trs |^2 \nabla \Trs\|_{L^{\frac{7}{6}}} \|\nabla (-\Delta)^{-\frac{1}{2}}\Trs \|_{L^7} 
\leq C \|\nabla \Trs \|_{L^{\frac{7}{2}}}^3  \|\Trs \|_{L^7} \\
&\leq C \| \Trs \|_{H^{\frac{10}{7}}}^3  \|\Trs \|_{L^7} \,.
\end{align}
Therefore, $I : H^{\frac{10}{7}} \to \RR$, and by passing $k\to\infty$ in \eqref{pmcv} and using of weak convergence $\mu_k\to\mu$ on $H^\frac{10	}{7}$ (see Lemma 
\ref{l:pjc}) and the boundedness of $\psi_R(I(\Trs)) I(\Trs)$
we obtain 
\begin{align}\label{alih}
\frac{A_{-\frac{1}{2}}}{2} - C \left( \frac{1}{R^\frac{2}{5}} + \frac{1}{R^{\frac{1}{6}}}\right)
\leq \int_{L^2}\psi_R(I(\Trs)) I(\Trs) \mu(d\Trs)\leq \frac{A_{-\frac{1}{2}}}{2} + C \left( \frac{1}{R^\frac{2}{5}} + \frac{1}{R^{\frac{1}{6}}}\right).
\end{align} 
Finally, by \eqref{fgnh} one has 
\begin{equation}\label{siei}
|\psi_R(I(\Trs)) I(\Trs)| \leq 
|I(\Trs)| \leq C ( \|\nabla \Trs \|_{L^4}^{\frac{10}{3}}  \|\Trs \|_{L^2}^{\frac{2}{3}} + \|\Trs\|_{\frac{3}{2}}^2)
\end{equation}
and by \eqref{est_H1_lim_meas} the right hand side is $\mu$ integrable. Since $\psi_R(I(\Trs)) I(\Trs) \to I(\Trs)$ everywhere as $R \to \infty$, 
by the dominated convergence theorem, and \eqref{alih} implies  
 \eqref{est_L2_lim_meas}.
\end{proof}

\section{Qualitative properties}\label{section_QualProperties}

In this section we complete the proof of the main result, Theorem \ref{thm:main}
by showing that parts 6 and 7 holds  true. 

In particular, 
we show that the the distributions via $\mu$ of the functionals below admit densities with respect to the Lebesgue measure on $\RR.$ 
\begin{align*}
E_{-\frac{1}{2}}(\Trs) &=\frac{1}{2}\|\Trs\|_{H^\frac{-1}{2}}^2,\\
M(\Trs) &=\frac{1}{2}\|\Trs\|^{2}.
\end{align*}  
Also, using other conservation laws of the SQG equation, we show the infinite-dimensional nature of the measure $\mu$. The proofs follow general framework
developed for analogous problems, however the adaptation is not straightforward. 
Since our smoothing operator is not Laplacian, but bi-Laplacian, we lost several 
important properties. For example,  unlike $\langle  -\Delta \Trs, f(\Trs)\rangle \geq 0$
for any increasing function $f$, it is not clear that 
 $\langle \Delta^2 \Trs, f(\Trs)\rangle$ is bounded from below for any $\Trs \in H^2$ and for sufficiently many functions $f$. This obstacle was solved by introducing
the $p$-Laplacian to the equation, that lead to stronger moment bounds, see Theorem \ref{TheoremPassagetotheLimit}. In such case, after nontrivial integration by parts we can show that  $\langle \Delta^2 \Trs, f(\Trs)\rangle$ is bounded from below if $f$ has bounded derivative up to fourth order. 
Also, compared to Euler equation we have to choose differently the set of functions $f$. 
 Let us provide details.

\begin{Thm}\label{theoremabsolutecontinuity}
Assume $A_0 < \infty$.
The laws of the functionals $M(\Trs)$ and $E_{-\frac{1}{2}}(\Trs)$ under $\mu$ are absolutely continuous with respect to the Lebesgue measure on $(0,\infty).$ 
\end{Thm}

Analogous statement can be proved for other invariants, but it is more technically involved, and we decided to skip it for the clarity of presentation. 

To obtain the following result in a cleaner form, we redefine the measure $\mu$ constructed in Theorem 
\ref{TheoremPassagetotheLimit} so that it does not have an atom at the origin. 

By \eqref{est_L2_lim_meas} and $A_{-\frac{1}{2}} \neq 0$ 
one has $\mu(\{0\}) < 1$, and therefore $S := \mu(H^2 \setminus \{0\})>0$. Define the probability measure 
\begin{equation}\label{dtm}
\tilde{\mu}(A) =  \frac{\mu(A\setminus \{0\})}{S}  \,.
\end{equation}
The SQG equation preserves the $L^2$ norm of solutions, and therefore it 
preserves the set $\textrm{supp}(\mu) \setminus \{0\}$. The invariance of $\tilde{\mu}$ comes readily from the invariance  of $\mu$.

\begin{Thm}\label{theoremdimension}
If $A_0 < \infty$, then 
he measure $\tilde{\mu}$ is infinite-dimensional in the sense that  if 
$K\subset H^1$ is a compact set of finite Hausdorff dimension, then
$\mu(K)=0.$
\end{Thm}

\begin{proof}[Proof of Theorem $\ref{theoremabsolutecontinuity}$]

Let $F(\Trs)$ be either $M(\Trs)$ or $E_{\frac{-1}{2}}(\Trs)$.
Thanks to the Portmanteau theorem, it suffices to prove the theorem for 
measures $\mu_\alpha$ with bounds that are uniform in $\alpha.$
Also, according to the non-negativity of $F$, our analysis shall be reduced to the interval $[0,\infty).$

\noindent
\textbf{Step $1:$ The pilot relation.} Fix any $f  \in C^\infty_0$ and define
\begin{align*}
\Phi_\delta(x)=\frac{1}{\sqrt{2\delta}}\int_{-\infty}^\infty f(y)e^{-|x-y|\sqrt{2\delta}}dy=\frac{1}{\sqrt{2\delta}}\left(\int_{-\infty}^xf(y)e^{-(x-y)\sqrt{2\delta}}dy+\int_x^{\infty}f(y)e^{(x-y)\sqrt{2\delta}}dy\right).
\end{align*}
and then
\begin{align*}
\Phi_\delta'(x)=\int_x^{\infty}f(y)e^{(x-y)\sqrt{2\delta}}dy-\int_{-\infty}^xf(y)e^{-(x-y)\sqrt{2\delta}}dy.
\end{align*}
Computing the second derivative of $\Phi_\delta$, we obtain that
\begin{align}
\frac{1}{2}\Phi_\delta''+f=\delta \Phi_\delta.\label{equ_Phi_lamda}
\end{align}
Since $\Phi_\delta$ is bounded uniformly in $\delta$ (as $f$ is compactly supported),
for every $x$, $\delta \Phi_\delta(x) \to 0$ as $\delta \to 0$ and 
\begin{align}
\Phi'_\delta(x) &\to\int_x^{\infty}f(y)dy -\int_{-\infty}^xf(y)dy \quad\text{as $\delta\to 0$},\label{CONV_phi_prim}\\
\Phi''_\delta(x) &\to -2f(x)\quad\text{as $\delta\to 0$}.\label{CONV_phi_sec}
\end{align}
Assume that $\Trs$ is a solution of \eqref{sqgs} with $\Trs(0)$ distributed as $\mu_\al$, and therefore 
$\Trs(t)$ is distributed as $\mu_\al$ for any $t \geq 0$. Note that by Theorem \ref{TheoremGWP}, $\Trs$ is a global solution and to simply the notation, we 
will not indicate explicitly the dependence of $\Trs$ on $\alpha$.
Denote $\E_{\mu_\alpha}$ to integral with respect to the measure 
$\mu_\alpha(d\Trs)$.

Let us apply the It\^o formula to $\Phi_\delta(F(\Trs))$, take the expectation, and use $\nabla_\Trs F(\Trs; \bfU\cdot\nabla\Trs)=0$
and stationarity of $\Trs(t)$ to obtain 
\begin{multline}\label{mlf}
\E_{\mu_\alpha}\left(\Phi_\delta'(F(\Trs))\left[\nabla_\Trs F(\Trs;\Delta^2 \Trs - \nabla(|\nabla \Trs|^2 \nabla \Trs)) 
+\frac{1}{2}\sum_{m} a_m^2\nabla_\Trs^2F(\Trs;e_m,e_m)\right]\right)
\\
+\frac{1}{2}\E_{\mu_\alpha}\left(\Phi''_\delta(\Trs)\sum_{m} a_m^2(\nabla_\Trs F(\Trs,e_m))^2\right)=0.
\end{multline}
If $F(\Trs) = M(\Trs)$, then 
\begin{equation}
\nabla_\Trs F(\Trs;\Delta^2 \Trs - \nabla(|\nabla \Trs|^2 \nabla \Trs))  = 
\|\Trs\|_{H^2}^2 + \|\Trs\|_{W^{1, 4}}^4 
\end{equation}
and if $F(\Trs) = E_{\frac{-1}{2}}(\Trs)$, then 
\begin{equation}
\nabla_\Trs F(\Trs;\Delta^2 \Trs - \nabla(|\nabla \Trs|^2 \nabla \Trs))  = I(\Trs)
\end{equation} 
where $I$ is defined by \eqref{dith}.

By \eqref{alp}, $\mu_\alpha$ is supported on $H^2$. Since $f\in C^\infty_0$, 
$|\Phi'_\delta|, |\Phi''_\delta|$ are bounded independently of $\delta$, 
 we can use the Lebesgue dominated convergence theorem to pass $\delta \to 0$
 in \eqref{mlf} and obtain for $F = M$
\begin{multline}\label{amf}
\E_{\mu_\alpha}\left(\left[\int_{F(\Trs)}^\infty f(y)dy-\int_{-\infty}^{F(\Trs)}f(y)dy\right]\left[\|\Trs\|_{H^2}^2 + \|\Trs\|_{W^{1, 4}}^4  
+\frac{1}{2}\sum_{m} a_m^2\nabla_\Trs^2F(\Trs;e_m,e_m)\right]\right)\\
-\E_{\mu_\alpha}\left(f(F(\Trs))\sum_{m} a_m^2|\nabla_\Trs F(\Trs,e_m)|^2\right)=0.
\end{multline}
and for $F = E_{-\frac{1}{2}}$ we just replace $\|\Trs\|_{H^2}^2 + \|\Trs\|_{W^{1, 4}}^4$ by $I(\Trs)$.

By a standard approximation argument combined with the Lebesgue dominated convergence theorem, we can
extend \eqref{amf} to $f=\chi_\Gamma$ being the characteristic function 
of a Borel set $\Gamma \subset \RR$. Then, $F \geq 0$ and 
\eqref{igi} imply that 
there is $C$ independent of $\alpha$  and $\Gamma$ such that 
\begin{align}
\E_{\mu_\alpha}&\left(\left[\int_{F(\Trs)}^\infty \chi_\Gamma(y)dy -\int_{-\infty}^{F(\Trs)}\chi_\Gamma(y)dy\right]\left[\|\Trs\|_{H^2}^2 + \|\Trs\|_{W^{1, 4}}^4  
+\frac{1}{2}\sum_{m} a_m^2\nabla_\Trs^2F(\Trs;e_m,e_m)\right]\right)\nonumber\\
&\leq
\left(\int_{0}^\infty \chi_\Gamma(y)dy \right)
\E_{\mu_\alpha} \left[\|\Trs\|_{H^2}^2 + \|\Trs\|_{W^{1, 4}}^4  
+\frac{1}{2}\sum_{m} a_m^2 \right]\nonumber\\
&\leq C\int_{-\infty}^\infty\chi_{\Gamma}(y)dy\leq C\ell(\Gamma) \,,
\end{align}
where $\ell(\Gamma)$ denotes the Lebesgue measure of $\Gamma$. If 
$F = E_{-\frac{1}{2}}$, one obtains a similar bound by using \eqref{siei}
and \eqref{igi}.

By \eqref{amf},
\begin{equation}\label{absolute_continuity:pilot_relation}
\E_{\mu_\alpha}\left(\chi_\Gamma(F(\Trs))\sum_{m} a_m^2|\nabla_\Trs F(\Trs,e_m)|^2\right) 
\leq C \ell(\Gamma).
\end{equation}
In the remaining part of the proof, we estimate the left hand side from below.

\noindent
\textbf{Step $2$: Absolute continuity on $(0, \infty).$}
Recall that $F$ is either $E_{\frac{-1}{2}}(\Trs)$ or $M(\Trs)$, then $\nabla_\Trs F(\Trs; e_m)$ is respectively
\begin{align*}
\lambda_m^{-1}(\Trs,e_m)\ \ \ \textrm{or} \ \ \ \ (\Trs,e_m).
\end{align*}
First,  focus on  $F (\Trs) = E_{\frac{-1}{2}}(\Trs)$.
 For any $\delta > 0$ denote 
 $\Gamma_\delta = \{\|\Trs\|_{-1}^2\geq \delta\} \cap \{\|\Trs\|^2\leq 1/\delta\}$.  
 
Let  $\bar{a}_N = \min \{| a_m|, \ 1\leq|m|\leq N\}$ and recall that $0 < \lambda_1 \leq \lambda_2 \leq \cdots \leq \lambda_N \leq \cdots$
are eigenvalues of $(-\Delta)$. 
Let $P_N$ be the projection on the space spanned by the first $N$ eigenfunctions of the Laplacian. Then, by the inverse Poincar\' e inequality,
for any $\Trs \in \Gamma_\delta$ one has
\begin{align*}
\sum_{m} a_m^2\lambda_m^{-2}(\Trs,e_m)^2 &=\sum_{|m|\leq N} a_m^2\lambda_m^{-2}(\Trs,e_m)^2+\sum_{|m|\pg N} a_m^2\lambda_m^{-2}(\Trs,e_m)^2\\
&\geq \bar{a}_N^2\lambda_N^{-1}\|P_N\Trs\|_{-1}^2 \geq \bar{a}_N^2\lambda_N^{-1}(\|\Trs\|_{-1}^2-\lambda_N^{-1}\|(I - P_N)\Trs\|^2)\\
&\geq \bar{a}_N^2\lambda_N^{-1} \left(\delta-\frac{1}{\delta \lambda_N}\right) \,.
\end{align*}
Fix any $\bar{\delta} > 0$ and any Borel set  $\Gamma^* \subset (\bar{\delta}, \infty)$.  For any $\delta \in (0, \bar{\delta})$, using \eqref{amf}
 the left hand side of \eqref{absolute_continuity:pilot_relation} can be estimated as
\begin{equation}
\E_{\mu_\alpha}\left(\chi_{\Gamma^*}(F(\Trs))\sum_{m} a_m^2|\nabla_\Trs F(\Trs,e_m)|^2\right)\geq  \bar{a}_N^2 \lambda_N^{-1} \left(\delta-\frac{1}{\delta \lambda_N}\right) \mu_\alpha(F^{-1}(\Gamma^*) \cap \Gamma_\delta).
\end{equation}
Since the sequence $\lambda_N$ increases to infinity, we can find $N$ such that 
$\delta-\frac{1}{\delta \lambda_N}\pg 0$, and therefore by \eqref{absolute_continuity:pilot_relation}
\begin{equation}
\mu_\alpha(F^{-1}(\Gamma^*) \cap \Gamma_\delta) 
\leq \frac{C\lambda_N}{\bar{a}_N^2\left(\delta-\frac{1}{\delta \lambda_N}\right) }\ell(\Gamma^*).
\end{equation}
Using the Portmanteau theorem, we pass to the limit $\alpha\to 0$ and  obtain
\begin{equation}
\mu(F^{-1}(\Gamma^*) \cap \Gamma_\delta) \leq \frac{C \lambda_N}{\bar{a}_N^2\left(\delta-\frac{1}{\delta \lambda_N}\right) }\ell(\Gamma^*).
\end{equation}
Consequently, if $\ell(\Gamma^*) = 0$, then for any $\delta \in (0, \bar{\delta})$
\begin{equation}
\mu(F^{-1}(\Gamma^*) \cap \Gamma_\delta)  = 0.
\end{equation}
Since $\bigcup_{\delta \in (0, \bar{\delta})} \Gamma_\delta = L^2 \setminus \{0\}$
and $0 \not \in F^{-1}(\Gamma^*)$, one obtains
\begin{equation}
\mu_\alpha(F^{-1}(\Gamma^*)) = \mu_\alpha(F^{-1}(\Gamma^*)  \setminus \{0\})) = 0 
\end{equation}
and  the claimed continuity follows.

Next, we focus on $F(\Trs) = M(\Trs)$. 
By the Cauchy-Schwartz inequality,
\begin{align*}
\|\Trs\|^{2}&=(\Trs,\Trs) =\sum_{m} (\Trs, e_m)^2 = \sum_{|m|\leq N} (\Trs, e_m)^2 + \sum_{|m|> N} (\Trs, e_m)^2 \\
&\leq\frac{\|\Trs\|}{\bar{a}_N}\left(\sum_{|m|\leq N} a_m^2 (\Trs, e_m)^2\right)^{\frac{1}{2}}+\|\Trs\|\left(\sum_{|m|> N}(\Trs, e_m)^2\right)^{\frac{1}{2}}.
\end{align*}
Also, 
\begin{align*}
\sum_{|m|> N}(\Trs, e_m)^2\leq \frac{1}{\lambda_N^2}\|\Trs\|_1^2 \,,
\end{align*}
and therefore
\begin{align}\label{est_gradient_Fp}
\overline{a}_N \left(\|\Trs\|- \frac{1}{\lambda_N}\|\Trs\|_1 \right)\leq \left(\sum_{m} a_m^2 (\Trs,e_m)^2\right)^\frac{1}{2}.
\end{align}
For any $\epsilon>0$ denote $I_\epsilon=\{\|\Trs\| \geq \epsilon, \ \ \|\Trs\|_1\leq\frac{1}{\epsilon}\}$. Then, for any $\Trs \in I_\varepsilon$
\begin{align}
\overline{a}_N \left(\epsilon - \frac{1}{\bar{\lambda}_N \varepsilon} \right)\leq \left(\sum_{m} a_m^2(\Trs, e_m)^2\right)^\frac{1}{2}.\label{est_gradient_Fp_epsilon}
\end{align}
Fix $\bar{\varepsilon}>0$ and a Borel set $\Gamma  \subset [\bar{\varepsilon}, \infty)$. 
Since $\Trs$ is distributed as $\mu_\al$, and $M(\Trs) \in \Gamma$ implies 
$\|\Trs\| \geq \bar{\varepsilon}$, then for any $\varepsilon \in (0, \bar{\varepsilon})$ 
\begin{align*}
\mu_\al(M^{-1}(\Gamma))=  \mu_\alpha(\{M(\Trs)\in \Gamma\}\cap \{\Trs\in I_\varepsilon\}) + \mu_\alpha \left(\{M(\Trs)\in \Gamma\} \cap \left \{\|\Trs\|_1\geq \frac{1}{\varepsilon}\right\}\right)=I+II.
\end{align*}  
Using the Chebyshev inequality and  \eqref{igi}, we obtain
$$
II\leq C\frac{A_0}{2}\varepsilon^2.
$$
Since $\lambda_N\to\infty$ as $N \to \infty$, we can suppose that $(\varepsilon-\frac{1}{\varepsilon \lambda_N})>0$. By  \eqref{absolute_continuity:pilot_relation} and \eqref{est_gradient_Fp_epsilon},
$$
I\leq \frac{C}{\bar{a}_N (\varepsilon-\frac{1}{\lambda_N \varepsilon})}\ell (\Gamma) \,.
$$
Consequently Portmanteau theorem yields
\begin{align*}
\mu(M^{-1}(\Gamma))\leq\frac{C}{\bar{a}_N (\varepsilon-\frac{1}{\varepsilon \lambda_N})}\ell (\Gamma)+ \frac{A_0}{2}\varepsilon^2
\end{align*}
and the rest of the proof follows as in the previous case. 
\end{proof}

Next, let us turn to the proof of Theorem \ref{theoremdimension}. 

\begin{proof}[Proof of Theorem $\ref{theoremdimension}$]
For any positive integer $k$ denote 
\begin{align*}
F_k(\Trs)= \frac{1}{|\TT|}\int_{\TT}f_k(\Trs(x))dx,
\end{align*}
where $f_k$ is a smooth function on $\RR$. Then, by $\nabla\cdot\bfU=0,$ we have for any $\Trs \in H^2$
\begin{equation}\label{cffs}
(f_k'(\Trs),\bfU\cdot\nabla\Trs)=0.
\end{equation}
Therefore, the functionals $F_k$ are conservation laws for \eqref{sqg}.
Fix $n$ and functions $(f_k)_{k= 1}^n$ on $\RR$ such that 
\begin{itemize}
\item[(i)] There is a constant $C$ independent of $k$ such that $|f^{(p)}(z)| \leq C$ for $p \in \{0, \cdots, 4\}$, that is, the sequence $(f_k)$ has uniformly (in $k$) bounded derivatives up to fourth order. Note that the bound can depend on $n$.
\item[(ii)] $f_k(0) = 0$ for each $k$. 
\item[(iii)] $f_1 \geq 0$ and $f_1 > 0$ on $(-\delta^*, \delta^*) \setminus\{0\}$ for some $\delta^* > 0$. 
\item[(iv)] If for some 
$v \in \RR^n$ and some continuous function $m : \TT \to \RR$ with zero mean 
one has 
\begin{align}\label{ipf}
\sum_{i=1}^n v_if_i'(m(x)) = \textrm{Const},\ \ \textrm{for all } x\in\TT\,, \textrm{ then } 
 v_i = 0 \quad \textrm{for each $i$,} \quad  \textrm{ or } \quad m \equiv 0.
\end{align}
\end{itemize}
For any $n \geq 1$,  such $(f_k)_{k = 1}^n$ indeed exists. For example let 
$f_k$ be smooth functions, compactly supported on $[-2, 2]$ and $f_k(z) = z^{k+1}$ on $(-1, 1)$. In addition, we assume that $f_1 \geq 0$. Clearly (i)--(iii) holds and it 
remains to verify (iv). Fix any zero mean continuous function $m \not \equiv 0$. Then the image of $m$ contains the interval $(-\delta, \delta)$ for some $\delta > 0$ and 
consequently 
\begin{equation}
\sum_{i=1}^n c_if_i'(z) = \textrm{Const} \qquad \textrm{ for all} \quad z \in (-\delta, \delta) \,.
\end{equation}
Since $f_k'$ are non-constant polynomials on $(-1, 1)$ one obtains that $v_i = 0$
for each $i$ as desired.

Since the second order derivatives of $f_k$ are bounded, then for any 
solution  $\Trs$  of \eqref{sqg} we can use Theorem \ref{intro_KS_Ito_theorem} and \eqref{cffs} to obtain
\begin{align}\label{ifg}
\E F_k(\Trs(t))+\alpha\E\int_0^t(f_k'(\Trs),\Delta^2\Trs - \nabla (|\nabla \Trs|^2 \nabla \Trs))ds=\E F_k(\Trs_0)+\frac{\al}{2}\sum_{m=1}^\infty a_m^2\E\int_0^t(f_k''(\Trs;e_m,e_m))ds.
\end{align}
Next, with a use of summation convention
\begin{align*}
(f_k'(\Trs),\partial^2_{ii}\partial^2_{jj}\Trs)=
-(f_k''(\Trs),  \partial_{i}\Trs \partial_i\partial^2_{jj}\Trs)
=
(f_k''(\Trs),\partial^2_{ii}\Trs\partial^2_{jj}\Trs)+(f_k^{(3)}(\Trs),(\partial_{i}\Trs)^2\partial^2_{jj}\Trs)=:(1)+(2).
\end{align*}
On the other hand,
\begin{align*}
(f_k'(\Trs),\partial^2_{ii}\partial^2_{jj}\Trs)=
-(f_k''(\Trs),  \partial_{i}\Trs \partial_i\partial^2_{jj}\Trs)
=
(f_k''(\Trs),(\partial_{ij}^2\Trs)^2)+(f_k^{(3)}(\Trs),\partial_i\Trs\partial_j\Trs\partial^2_{ij}\Trs)=: (3)+(4).
\end{align*}
Furthermore, 
\begin{align*}
(2)=-(f_k^{(4)}(\Trs),(\partial_i\Trs)^2(\partial_j\Trs)^2) - 2(f_k^{(3)}(\Trs),\partial_i\Trs\partial_j\Trs\partial^2_{ij}\Trs).
\end{align*}
We obtain
\begin{align*}
(f_k'(\Trs),\partial^2_{ii}\partial^2_{jj}\Trs)=
\frac{1}{3}[(1) + (2)] + \frac{2}{3}[(3) + (4)] = 
\frac{1}{3}\left((f_k''(\Trs), \partial_{ii}^2\Trs\partial^2_{jj}\Trs+2(\partial_{ij}^2\Trs)^2)-(f_k^{(4)}(\Trs),(\partial_i\Trs)^2(\partial_j\Trs)^2)\right) \,,
\end{align*}
and consequently
\begin{align}\label{doak}
(f_k'(\Trs),\Delta^2\Trs -  \nabla (|\nabla \Trs|^2 \nabla \Trs))=\frac{1}{3}\left((f_k''(\Trs),(\Delta\Trs)^2 + 2(D^2\Trs)^2) + (f_k''(\Trs)-f_k^{(4)}(\Trs),|\nabla\Trs|^4)\right)=: A_k(\Trs).
\end{align}
For any positive integer $n$, denote
\begin{align}
V_n(\Trs)=
\left(
\begin{array}{c}
F_1(\Trs)\\
F_2(\Trs)\\
\vdots\\
F_n(\Trs)
\end{array}
\right).\label{Krylov_vector}
\end{align}
By the It\^o formula and \eqref{cffs}, 
\begin{align}
V_n(\Trs)=V_n(\Trs_0)+\int_0^tx_sds+\sqrt{\al}\sum_m\int_0^ty_{m}(s)d W_m(s),\label{Krylov_vector_Ito_form}
\end{align}
and with $A_k$ defined in \eqref{doak} one has
\begin{align*}
x_s &=-\alpha\left(\begin{array}{c}
A_1(\Trs)\\
A_2(\Trs)\\
\vdots\\
A_n(\Trs)
\end{array}
\right)+\frac{\alpha}{2}\sum_m a_m^2
\left(\begin{array}{c}
\nabla_\Trs^2F_1(\Trs;e_m,e_m)\\
\nabla_\Trs^2F_2(\Trs;e_m,e_m)\\
\vdots\\
\nabla_\Trs^2F_n(\Trs;e_m,e_m)
\end{array}
\right)\\
&=-\alpha\left(\begin{array}{c}
A_1(\Trs)\\
A_2(\Trs)\\
\vdots\\
A_n(\Trs)
\end{array}
\right)+\frac{\alpha}{2}\sum_m a_m^2
\left(\begin{array}{c}
(f_1''(\Trs);e_m,e_m)\\
(f_2''(\Trs);e_m,e_m)\\
\vdots\\
(f_n''(\Trs);e_m,e_m)
\end{array}
\right)=:-\al A(\Trs)+\frac{\al}{2}B(\Trs),\\
y_{m} &= a_m
\left(\begin{array}{c}
(f_1'(\Trs),e_m)\\
(f_2'(\Trs),e_m)\\
\vdots\\
(f_n'(\Trs),e_m)
\end{array}
\right).
\end{align*}
Let $y_{m}^i = a_m (f'_i(\Trs), e_m)$ be the $i$th component of $y_m$. Denote by $M$ the $n \times n$ matrix with entries
\begin{align*}
M_{i,j}=\sum_{m}y_{m}^iy_{m}^j=  \sum_m a_m^2(f_i'(\Trs),e_m)(f_j'(\Trs),e_m)
\end{align*}
and note that $M$ depends on $t$, but is independent of $x$. 

Since $f'_k$, $f''_k$, and $f^{(4)}_k$ are bounded, by \eqref{doak} and \eqref{igi} one has  for any $\alpha \in (0, 1)$
\begin{equation}\label{bxu}
\E \int_0^t |x_s| + \alpha \sum_{m = 1}^\infty |y_m(s)|^2 ds \leq 
C \alpha \E \sum_{k = 1}^n\int_0^t  \|\Trs\|_{H^2}^2 + \|\nabla \Trs\|^4 + 1 +  
 \sum_{m = 1}^\infty  a_m^2 
ds \leq C \,,
\end{equation}
where $C$ is independent of $\alpha$. 
Then, \cite[Theorem 7.9.1]{KS12} and \eqref{bxu} (bound on $|x_s|$) imply
for any bounded measurable function $g$ the 
Krylov's estimate
\begin{equation}
\E_{\mu_\alpha}\int_0^1(\det M)^{1/n}g(V_n)dt\leq C_n \|g\|_{L^n}.
\end{equation} 
Let $\mathcal{B} \subset \RR^n$ be a Borel set and denote  $g=\chi_{\mathcal{B}}$
the indicator function of $\mathcal{B}$. 
Then, since $\mu_\alpha$ is an invariant measure
\begin{equation}\label{kess}
\int (\det M(\bar{\Trs}))^{1/n}\chi_{\mathcal{B}}(V_n(\bar{\Trs})) d \mu_\alpha(\bar{\Trs}) = 
\E_{\mu_\alpha}\int_0^1(\det M(\Trs(s)))^{1/n}\chi_{\mathcal{B}}(V_n(\Trs(s)))ds \leq C_nA_0(\ell_n(\mathcal{B}))^{\frac{1}{n}}.
\end{equation}
For any integer $k > 0$ denote $\mathcal{B}_k = \mathcal{B} \cap B_k^c$, where
$B_k \subset \RR^n$ is a ball of radius $\frac{1}{k}$ centred at the origin, and $B_k^c$ is 
its complement.  Note that 
$\mathcal{B} = \{0\} \cup \bigcup_k \mathcal{B}_k$ and by \eqref{kess} for any $k > 0$
  \begin{equation}\label{kes}
\int (\det M(\bar{\Trs}))^{1/n}\chi_{\mathcal{B}_k}(V_n(\bar{\Trs})) d \mu_\alpha(\bar{\Trs}) \leq C_nA_0(\ell_n(\mathcal{B}))^{\frac{1}{n}}.
\end{equation}

\noindent
\textbf{Estimate on $\mathbf{\det M}$.} The matrix $M$ is clearly a non-negative symmetric $n\times n$ matrix. 
We show that for any $\epsilon > 0$, $M$ is positively bounded from below 
outside the ball in $H^1$ of radius $\epsilon\pg 0$ centered at the origin. 
Observe that $M$ is an infinite sum of non-negative matrices $M^m$ with coefficients
\begin{equation}
M^m_{i,j}=  a_m^2(f_i'(\Trs),e_m)(f_j'(\Trs),e_m).
\end{equation}
Then, for any vector $v=(v_1,...,v_n)\in \RR^n$, we have
\begin{align*}
(v, Mv) &=\sum_{m\geq 0}(v, M^m v)=\sum_{m\geq 0}\sum_{1\leq i, j\leq n}M^m_{i,j}v_iv_j\\ 
&=\sum_{m\geq 0} a_m^2\sum_{1\leq i, j\leq n}v_i v_j (f_i'(\Trs(x)), e_m(x))  (f_j'(\Trs(x)), e_m(x)) \\
&=\sum_{m\geq 0} a_m^2\left(\sum_{j=1}^n v_j  (f_j'(\Trs(x)), e_m(x))\right)^2=
\sum_{m\geq 0} a_m^2\left(\sum_{j=1}^n v_jf_j'(\Trs(x)),e_m \right)^2.
\end{align*}
Suppose that $(v, Mv)=0$ for some $v \neq 0$, since $ a_m\neq 0$ for all $m$, 
\begin{align*}
\left(\sum_{j=1}^n v_jf_j'(\Trs),e_m \right) = 0 \qquad \textrm{for all } m \geq 0.
\end{align*}
Hence, the function $x\mapsto\sum_{j=1}^n v_jf_j'(\Trs(x))$ is constant, that is,  there is $C$ such that 
\begin{align*}
\sum_{j=1}^n v_jf_j'(\Trs(x))=C \qquad  \textrm{for all } x\in \TT.
\end{align*}
By the independence property \eqref{ipf},  
either $v_k = 0$ for each $k$ or $\Trs \equiv C$. Since $v \neq 0$,  the latter 
property holds and since $\Trs$ has zero mean we have $\Trs \equiv 0$.  

Therefore if $\Trs \not \equiv 0$, then  $\textrm{det} (M) > 0$.
Next, denote $I_\epsilon=\{\|\Trs\| \geq\epsilon,\ \ \|\Trs\|_{H^1}\leq\frac{1}{\epsilon}\}$, and note that $I_\epsilon$  is compact in $L^2$. Indeed, 
if $(\Trs_j)_j \subset I_\epsilon$, then $\|\Trs_j\|_{H^1} \leq \frac{1}{\varepsilon}$, 
and therefore there exists subsequence, still denoted
$(\Trs_j)_j$,  converging to $\Trs_\infty$ weakly in $H^1$ and strongly in $L^2$.  Weak lower semi-continuity, and strong continuity of norms 
yield $\Trs \in I_\varepsilon$ as desired.  
 
By smoothness of $f_k$, uniform boundedness of $f'$, and $A_0 < \infty$, the map 
$\Trs \mapsto M(\Trs) : H^{1} \to \RR^{n\times n}$ is continuous, and consequently $\Trs \mapsto \det M(\Trs): H^{1}\to [0, \infty)$ is continuous
as well. Since $\det M > 0$  on the compact set $I_\varepsilon$, then
 $\det M(\Trs)\geq c_\varepsilon > 0$ on $I_\varepsilon$.

\noindent
\textbf{Conclusion.} For the Borel set $\mathcal{B}_k$ fixed in \eqref{kes}, 
and $V_n$ defined in 
\eqref{Krylov_vector} one has for any $k$
\begin{align*}
\mu_\alpha(\{V_n(\Trs)\in \mathcal{B}_k\})\leq 
\mu_\alpha(\{V_n(\Trs)\in \mathcal{B}_k\}\cap\{\Trs \in I_\epsilon\})
+
\mu_\alpha\left(\{V_n(\Trs)\in \mathcal{B}_k\}\cap \left\{ \Trs  \in I_\varepsilon^c \right\}\right)
=I+II.
\end{align*}
Since $\det M(\Trs)\geq c_\epsilon,$ on $I_\epsilon$,  \eqref{kes}
yields
\begin{align*}
I\leq A_0C_nc_\epsilon^{-\frac{1}{n}} (\ell_n(\mathcal{B}))^{\frac{1}{n}}.
\end{align*}
Next, since $(f_k)$ are uniformly, globally Lipschitz and $f_k(0) = 0$, there exists $L \geq 0$ such that 
$|f_k(z)| \leq L |z|$ for any $k$ and any $z \in \RR$.  Then, for any $\Trs$ with 
$V_n(\Trs) \in \mathcal{B}_k$ one has $|V_n(\Trs)| \geq \frac{1}{k}$, and therefore
\begin{equation}
\frac{1}{k} \leq C \max_j |F_j(\Trs)| \leq C \max_j \int_{\TT} |f_j(\Trs(x))| dx 
\leq C\|\Trs\| \,,
\end{equation}
where we used Jensen's inequality in the last estimate. Without loss of generality assume $C \geq 1$.
Thus,  by \eqref{igi} and the Chebyshev inequality, if $\varepsilon < \frac{1}{C k}$, then 
\begin{align}
II &\leq  \mu_\alpha\left(\{V_n(\Trs)\in \mathcal{B}_k\}\cap \left\{ \|\Trs\|_{H^1} \geq \frac{1}{\varepsilon} \right\}\right) + 
 \underbrace{\mu_\alpha\left(\{V_n(\Trs)\in \mathcal{B}_k\}\cap \left\{ \|\Trs\| \leq \varepsilon \right\}\right)}_{ = 0}
\\
&\leq 
  \mu_\alpha\left(\left\{ \|\Trs\|_{H^1} \geq \frac{1}{\varepsilon} \right\}\right) 
\leq 
\frac{A_0}{2}\epsilon^2.
\end{align}
Gathering these estimates, we arrive at
\begin{equation}\label{absolcontVn}
\mu_\alpha(\{V_n(\Trs)\in \mathcal{B}_k\})\leq \frac{A_0}{2}\epsilon^2 + C_n c_\epsilon^{-\frac{1}{n}} (\ell_n(\mathcal{B}))^{\frac{1}{n}} 
\end{equation}
and by Portmanteau theorem,  $(\ref{absolcontVn})$ is valid with $\mu_\alpha$ replaced by the limiting measure $\mu$.
If $\ell_n(\mathcal{B}) = 0$, then since $\varepsilon > 0$ is arbitrary, we obtain 
$\mu(\{V_n(\Trs)\in \mathcal{B}_k\}) = 0$ for any $k > 0$. Taking the countable 
union in integer $k > 0$, we arrive to $\mu(\{V_n(\Trs)\in \mathcal{B} \setminus \{0\} \}) = 0$.

Since $f_1$ is non-negative and $f_1 > 0$ in a punctured neighbourhood of zero, then for any continuous, zero mean function $\Trs \not \equiv 0$ one has $F_1(\Trs) \neq 0$. Hence, 
\begin{equation}
\mu(\{V_n(\Trs)\in \mathcal{B} \} \setminus \{0\}) = 
\mu(\{V_n(\Trs)\in \mathcal{B} \setminus \{0\} \}) = 0\,. 
\end{equation}
By the definition of $\tilde{\mu}$,
\begin{equation}\label{rfe}
\tilde{\mu}(\{V_n(\Trs)\in \mathcal{B} \}) = \frac{\mu(\{V_n(\Trs)\in \mathcal{B} \} \setminus \{0\})}{S} = 0 
\end{equation}
for any $\mathcal{B}$ with $\ell_n(\mathcal{B}) = 0$.

Finally, we prove that $\tilde{\mu}$ is infinite dimensional. 
Let $K \subset L^2$ be a compact set with finite Hausdorff dimension
$\dim_H (K) =: h$  and fix an integer $n>h$.
We claim that $V_n$ defined in \eqref{Krylov_vector}  is differentiable on $L^2$. Indeed, each component $F_k$ of $V_n$ satisfies
 \begin{align}
 |F_k''(\Trs;u,v)|= \frac{1}{|\TT|} \left|\int_{\TT}f_k''(\Trs)uvdx\right| \leq C \int_{\TT} |uv|dx \leq 
 C \|u\|\|v\| \,,
 \end{align}
 where we used that $f''_k$ is bounded. In particular, $V_n$ is  locally Lipschitz. 
 Since locally Lipschitz maps do not increase the Hausdorff dimension, 
$\dim_H(\mathcal{V}) \leq h < n$, where   
 $\mathcal{V} := V_n(K) \in \RR^n$, and therefore $\ell_n(\mathcal{V})=0$. 
 Then, by \eqref{rfe}
\begin{align*}
\tilde{\mu}(K)\leq \tilde{\mu}(\{V_n(\Trs) \in \mathcal{V}\})
= 0 
\end{align*}
as desired.
\end{proof}

\appendix

\section{Some facts on the fluctuation-dissipation approach for finite-dimensional Hamiltonian systems}\label{sec:ham}

In this section we elaborate  on the question that was raised in the introduction: Do the constructed invariant measure $\mu$ for \eqref{sqg} concentrates on the 
equilibria? Although we proved that the support of $\mu$ is infinite dimensional, it also known that the set of equilibria is also infinite dimensional: any solution of the equation 
\begin{equation}
 (-\Delta)^{\frac{1}{2}} \Phi = F(\Phi)
\end{equation}
is an equilibrium of \eqref{sqg}. Since every equilibrium is trivially a global solution,  there is a possibility that $\mu$ concentrates on the set of equilibria, and we 
did not construct any new solution. As mentioned above, we don't have a definite answer to this question, however we provide an example of a general system for which the measure arising from fluctuation dissipation method is not supported on equilibria. 

Since the SQG equation has a Hamiltonian structure, we will focus only on the Hamiltonian systems. There are several trivial examples in which the equilibria 
form a discrete set, and therefore  are of measure zero, for instance the cubic defocusing Schr\"odinger equation with only one equilibrium. The example closest to
SQG is 2D Euler equation, which has infinite dimensional manifold of equilibria with 
similar structure. However, whether the invariant measures for 2D Euler equation concentrate on equilibria is an open question, hence regularizing the problem 
might not help. 

Let us turn our attention to finite dimensional systems. Consider a $2n$-dimensional Hamiltonian system
\begin{align}\label{ohs}
\dot{x}=-\partial_yH(x,y),\quad \dot{y}=\partial_xH(x,y),
\end{align}
where $H:\RR^n\times\RR^n\to \RR$ is a smooth Hamiltonian function. It is well known that $f(H)dxdy$ is an invariant measure for the system, for any integrable smooth function $f$. 
We consider now the following fluctuation-dissipation model
\begin{align}\label{odefd}
dx=(-\partial_yH(x,y)-\alpha\partial_xH(x,y))dt + \sqrt{2\alpha}d\beta_1,\quad
 dy= (\partial_xH(x,y)-\alpha\partial_yH(x,y))dt + \sqrt{2\alpha}d\beta_2,
\end{align}
where $\beta_1, \beta_2$ are independent Brownian motions. 
Then, $e^{-H(x,y)}$ is a density of an invariant measure for \eqref{odefd}, since  
$e^{-H(x,y)}$ is solution of the 
Fokker-Plank equation
\begin{align}
\mathcal{L}\rho=\alpha\Delta\rho-\nabla\cdot\left[(\partial_yH(x,y)+\alpha\partial_xH(x,y),-\partial_xH(x,y)+\alpha\partial_yH(x,y))^T\rho\right]=0.
\end{align}
Thus $\mu(dxdy)=T^{-1}e^{-H(x,y)}dxdy$ is an invariant probability measure of 
\eqref{odefd}, were we denote $T=\int_{\RR^n\times\RR^n} e^{-H(x,y)}dxdy$ to be a partition function (normalization). Note that $T$ is finite if $H$ has appropriate increase at infinity. 
Observe that $\mu$  does not depend on $\alpha$, thus by passing  $\alpha  \to 0$, we see that $\mu$ is an invariant measure of \eqref{ohs}. 

If $H$  is constant on the unit ball of $\RR^n\times\RR^n$, then any point in that ball is an equilibrium of \eqref{ohs}, and therefore we have an open set of equilibria. On the other hand,  $\mu$ has positive density everywhere and in particular its support coincides with the whole space. There might be a possibility to apply this reasoning to infinite dimensional systems, but there are serious difficulties with coercivity of the dissipation. We leave this question open. 

\section{It\^o formula}\label{sec:Ito}

For reader's convenience we recall 
the It\^o formula in infinite dimensions, which used several times in the proofs of main results. 
We say that the equation $\eqref{sqgs}$  has the It\^o property on the triple $(H^{s-1},H^s,H^{s+1})$ if
\begin{enumerate}
\item for some $T> 0$,  \eqref{sqgs} has a unique solution on $[0, T)$
 for any data in $H^s$;
\item  
the process $h:=-\al(\Delta^2\Trs-\nabla(|\nabla\Trs|^2\nabla\Trs))-\bfU\cdot\nabla\Trs$ is $\mathcal{F}_t$-adapted and
\begin{align}
\mathbb{P}\left(\int_0^t(\|\Trs(r)\|_{s+1}^2+\|h(r)\|_{s-1}^2)dr< \infty, \ \ \forall\ t>0\right)=1,\ \ 
\sum_{m>0} a_m^{2}\lambda_m^s<\infty.\label{intro_hyp_KS}
\end{align}
\end{enumerate}
We have the following version of the It\^o's lemma proved in 
\cite[Section $A.7$]{KS12}. 

\begin{Thm}[\cite{KS12}]\label{intro_KS_Ito_theorem}
Let $F\in C^2(H^s,\RR)$ be a functional which is locally uniformly continuous, together with its first two derivatives, on $H^s$. Suppose that \eqref{sqgs} satisfies the It\^o property on $(H^{s-1},H^s,H^{s+1})$ and that $F$ satisfies the following conditions:
\begin{enumerate}
\item There is a function $K:\RR_+\to\RR_+$ such that
\begin{equation}\label{intro_KS_Ito_F1}
|F'(\Trs;v)|\leq K(\|\Trs\|_{s})\|\Trs\|_{{s+1}}\|v\|_{{s-1}},\ \ \ \Trs\in H^{s+1},\ \ v\in H^{s-1}.
\end{equation}
\item For any sequence $\{w_k\}\subset H^{s+1}$ converging toward $w\in H^{s+1}$ and any $v\in H^{s-1}$, we have
\begin{equation}\label{intro_KS_Ito_F2}
F'(w_k;v)\to F'(w;v),\ \ as\ \ k\to\infty.
\end{equation}
\item The solution $\Trs$ of \eqref{sqgs} satisfies
\begin{equation}\label{intro_KS_Ito_F3}
\sum_{m} a_m^2\E\int_0^t|F'(\Trs;e_m)|^2ds <\infty\ \ \ \ for\ all \ t>0.
\end{equation}
\end{enumerate}
Then we have
\begin{align}\label{intro_KS_Ito}
F(\Trs(t))=F(\Trs(0))+\int_0^t\left(F'(\Trs(s);h(s))+\frac{\al}{2}\sum_{m} a_m^2F''(\Trs(s); e_m, e_m)\right)ds \nonumber\\
+\sqrt{\al}\sum_{m} a_m\int_0^tF'(\Trs(s);e_m)d W_m(s).
\end{align}
In particular,
\begin{equation}
\E F(\Trs(t))=\E F(\Trs(0))+\int_0^t\E\left(F'(\Trs(s); h(s))+\frac{\al}{2}\sum_{m} a_m^2F''(\Trs(s); e_m, e_m)\right)ds.
\end{equation}
If one omits $(\ref{intro_KS_Ito_F3}),$ then we have the formula $(\ref{intro_KS_Ito})$ where $t$ is replaced by the stopping time $t\wedge \tau_n$, with
\begin{equation}
\tau_n=\inf\{t\geq 0,\ \|\Trs(t)\|_{s}> n\}, \ \ n\geq 0,
\end{equation}
with the convention $\inf\emptyset=+\infty.$
\end{Thm}

\section{Embedding \texorpdfstring{$L^2H^2 \cap W^{1,\frac{4}{3}}W^{-1, \frac{4}{3}} \hookrightarrow CH^{-\delta}$}{Lg}}\label{sec:imbed}

Although the parabolic embedding 
  $L^2H^2 \cap W^{1,\frac{4}{3}}W^{-1, \frac{4}{3}} \hookrightarrow CH^{-\delta}$
follows from standard 
arguments
we were not able to locate the proof in the literature. Hence, we outline the main steps
in this appendix. 
 
By 
\cite[Theorem 5.2]{Amann2000}, we have for any  $\theta > \frac{2}{3}$
\begin{equation}
L^2H^2 \cap W^{1,\frac{4}{3}}W^{-1, \frac{4}{3}} \hookrightarrow
C (H^2, W^{-1, \frac{4}{3}})_{\theta, p_\theta} \,,
\end{equation}
where $(H^2, W^{-1, \frac{4}{3}})_{\theta, p_\theta}$ is the real interpolation space and $p_\theta$
satisfies 
\begin{equation}
\frac{1}{p_\theta} = \frac{1-\theta}{2} + \frac{\theta}{\frac{4}{3}} 
\end{equation}
However, by \cite[(3.5)]{Amann2000}
for any $\varepsilon \in (0, 1)$ one has
\begin{equation}
(H^2, W^{-1, \frac{4}{3}})_{\theta, p_\theta} \hookrightarrow (H^{2 - \varepsilon}, W^{-1 - \varepsilon, \frac{4}{3}})_{\theta, p_\theta} =
(B^{2-\varepsilon}_{2, 2}, B^{-1- \varepsilon}_{\frac{4}{3}, \frac{4}{3}})_{\theta, p_\theta} \,,
\end{equation}
where $B^{s}_{p, q}$ is a Bessov space. From \cite[Theorem 6.4.5, (3)]{BerghLofstrom1976}
 and \cite[(3.5)]{Amann2000}
follows
\begin{equation}
(B^{2-\varepsilon}_{2, 2}, B^{-1- \varepsilon}_{\frac{4}{3}, \frac{4}{3}})_{\theta, p_\theta} = B^{(-3+\varepsilon)\theta + (2 - \varepsilon)}_{\frac{4}{2 + \theta}, \frac{4}{2 + \theta}}
= W^{(-3+\varepsilon)\theta + (2 - \varepsilon), \frac{4}{2 + \theta}} \,.
\end{equation}
Finally, by Sobolev embeddings 
\begin{equation}
W^{(-3+\varepsilon)\theta + (2 - \varepsilon), \frac{4}{2 + \theta}} \hookrightarrow W^{-\delta, 2} \,,
\end{equation}
where $\delta \leq \frac{5}{3} - \varepsilon + (3 - \varepsilon)\theta$. Since $\theta > \frac{2}{3}$ and $\varepsilon > 0$ can be chosen arbitrarily close to $\frac{2}{3}$
and $0$ respectively, one obtains 
\begin{equation}
L^2H^2 \cap W^{1,\frac{4}{3}}W^{-1, \frac{4}{3}} \hookrightarrow C W^{-\delta, 2}
\end{equation} 
for any $\delta > \frac{1}{3}$ as desired.

\end{document}